\documentclass[10pt,reqno,a4paper]{amsart}

\let\oldtocsection=\tocsection

\let\oldtocsubsection=\tocsubsection

\let\oldtocsubsubsection=\tocsubsubsection

\renewcommand{\tocsection}[2]{\hspace{0em}\oldtocsection{#1}{#2}}
\renewcommand{\tocsubsection}[2]{\hspace{2em}\oldtocsubsection{#1}{#2}}
\renewcommand{\tocsubsubsection}[2]{\hspace{2em}\oldtocsubsubsection{#1}{#2}}

\usepackage{amsmath,amsthm,amsfonts,amssymb}

\usepackage[foot]{amsaddr}

\usepackage{marginnote}

\usepackage{fancyhdr}

\usepackage{graphicx}

\usepackage{pdfpages}

\usepackage{tikz}
\usetikzlibrary{matrix}
\usetikzlibrary{cd}

\usepackage{xcolor}

\numberwithin{figure}{section}
\numberwithin{equation}{section}

\newtheorem{thm}{Theorem}[section]
\newtheorem*{thmnonum}{Theorem}

\newtheorem{defn}[thm]{Definition}

\newtheorem{lmm}[thm]{Lemma}

\newtheorem{prp}[thm]{Proposition}

\newtheorem{cor}[thm]{Corollary}

\newtheorem{remark}[thm]{Remark}

\newtheorem{example}[thm]{Example}

\newcommand{\tei}{Teichm\"uller}
\newcommand{\qc}{quasiconformal}

\newcommand{\id}{\operatorname{id}}

\newcommand{\SV}{\operatorname{SV}}

\newcommand{\cyl}{\operatorname{cyl}}

\renewcommand{\Re}{\operatorname{Re\,}}
\renewcommand{\Im}{\operatorname{Im\,}}
\renewcommand{\mod}{\operatorname{mod\,}}

\DeclareMathOperator*{\arsinh}{arsinh}

\newcommand{\abs}[1]{\left| #1 \right|}

\newcounter{reminder}

\graphicspath{ {./images/}}

\title{On convergence of Thurston's iteration for entire functions with an infinite set of marked points}
\author{Konstantin Bogdanov$^1$}
\address{Institute of Mathematics of Polish Academy of Sciences, ul. Śniadeckich 8, 00-656 Warsaw, Poland}
\address{Saarland University, Mathematics and Computer Science, Campus E2 4, 66123 Saarbr\"ucken, Germany}
\email{konstantin.bogdanov@uni-saarland.de}

\begin{document}

\begin{abstract}	
	The goal of this note is to generalize Thurston's Topological Characterization of Rational Functions to the setting when both the covering degree and the set of marked points are infinite. A relevant class of branched coverings are transcendental entire functions with finitely many singular values whose orbits escape to (or, more generally, accumulate ``near'') $\infty$. Given a branched covering $f$ mimicking such post-singular behaviour, one wants to decide whether it is Thurston equivalent to an entire function. The answer is positive for a big class of entire function and generic escaping singular orbits. 
	
	As in the Thurston's theorem, the problem reduces to the study of the pull-back map $\sigma$ defined on the corresponding \tei\ space. But, unlike in the rational case, the space is infinite-dimensional and the branching structure near $\infty$ (which is essential singularity) is much more subtle depending on the family of functions under consideration. A general approach is possible for entire functions defined by \emph{asymptotic area property} introduced in the article. Roughly, it implies that asymptotic tracts fill all space near $\infty$ even when their range shrinks. 
	
	The main result provides a sufficient condition for existence in the \tei\ space of a $\sigma$-invariant subset which looks like a finite-dimensional compact: forgetting marked infinite tails of every orbit yields a compact set, while forgetting a long enough initial part of every orbit yields a small perturbation of the identity homeomorphism. The statement remains valid even in case of non-escaping unbounded singular orbits and allows to deduce existence of a fixed point of $\sigma$ in many relevant cases.	
\end{abstract}

\maketitle
	

\addtocontents{toc}{\protect\setcounter{tocdepth}{1}}

\footnotetext[1]{The author gratefully acknowledges partial support from National Science Centre, Poland, Grant OPUS21 ``Holomorphic dynamics, fractals, thermodynamic formalism'' 2021/41/B/ST1/00461, and from the ERC AdG grant 101097307.}

\section{Introduction}

Since its publication in 1993, Thurston's Topological Characterization of Rational Functions \cite{DH} has become one of the key tools in complex dynamics. Being a part of a more general picture along with the geometrization conjecture and Thurston's classification of surface homeomorphisms, the characterization theorem describes in the dynamical setting in which extent the geometrical structure is imposed by the underlying topology. More precisely, it provides a criterion of when a post-critically finite branched covering of the topological 2-sphere can be represented by a genuine rational function with ``the same'' post-critical behaviour.

There are two major directions for the generalization of this classical result.  The first one is by considering other classes of functions, for instance, transcendental entire and meromorphic functions. In this regard one should mention the result by Hubbard--Shishikura--Schleicher \cite{HSS}, generalizing the theory for the exponential family, and the PhD theses of Sergey Shemyakov \cite{SergeyThesis} and of Nikolay Prochorov \cite{KolyaThesis} encompassing more general families of entire functions. Another direction is by considering post-critically infinite dynamics. The corresponding generalization for hyperbolic rational functions is described by Cui--Tan \cite{Cui}. 

A subject of separate interest lies in the intersection of the two directions. A big class of examples of branched coverings of infinite degree are transcendental entire functions having finitely many singular values (or, in other words, branching values). Recall that such functions are said to be of finite type, or, equivalently, belong to the class $\mathcal{S}$. For simplicity reasons it is natural to restrict to the post-singular sets (defined as the closure of the union of orbits of singular values) having only one accumulation point. Thus, this can be either an attracting fixed point or $\infty$ with the latter case being more promising in terms of appearing phenomena (recall that $\infty$ is an essential singularity). Beyond that, on the level of parameter spaces, entire functions with escaping singular values may be viewed as an analogy to the complement of the Mandelbrot set. As parameter rays play an important role in study of the boundary of the Mandelbrot set, the ``generalized parameter rays'' for entire functions might help to improve our understanding of corresponding parameter spaces.

Thus, given a \emph{topological} entire function $f:\mathbb{C}\to\mathbb{C}$ with finitely many singular values all of which escape, we want to know when it is ``the same'' as a genuine entire function. On the formal level ``the sameness'' is called \emph{Thurston equivalence}, see Definition~\ref{defn:Thurston_equivalence}.

In the proof of Thurston's theorem one associates to the branched covering a ``pull-back map $\sigma$'' acting on the \tei\ space of the complement to the post-critical set. Then this branched covering is Thurston equivalent to a rational map if and only if $\sigma$ has a fixed point. Thus, the question of existence of a rational function which is Thurston equivalent to the covering is reduced to the study of the properties of $\sigma$. However, this procedure cannot be directly applied for an infinite post-singular set. The reason is that generally speaking $\sigma$ maps between different \tei\ spaces, see \cite[Lemma 3.5]{IDTT1} for a particular example.

This problem can be solved by considering a quasiregular function $f$, or even a quasiregular function of a special type, defined in a natural way using a capture. Let $f_0$ be a transcendental entire function with finitely many singular values (not necessarily escaping). We can choose the same number of (arbitrary) escaping orbits and a \qc\ map $\lambda:\mathbb{C}\to\mathbb{C}$ (called capture) equal to identity near $\infty$ such that the singular orbits of $f:=\lambda\circ f_0$ coincide with the chosen escaping orbits of $f_0$. Thus, $f$ is a quasiregular map mimicking some initially chosen post-singular behaviour. More generally, we can simply consider a map $f:=\lambda\circ f_0$ subject to the condition that singular orbits are eventually absorbed by the domain on which $\lambda$ is equal to identity.

Considering a quasiregular rather than just topological map $f$ as constructed above might seem a strong restriction. However, Lasse Rempe shows in \cite{LasseParaSpace} in a much greater generality that given two quasiconfomally equivalent functions, they are \qc ly conjugate on the set of points remaining (under iterations) in some neighbourhood of $\infty$. This implies that such functions $f$ as constructed above represent almost all of the apriori possible modes of escape inside of the corresponding parameter space.

Now, for expository reasons, before moving to a more detailed discussion, we formulate a rather restricted corollary of the results in this article. By the \emph{structurally finite} family we mean all entire function of the form
$$C+\int_0^z p(w)e^{q(w)}dw$$
where $p$ and $q$ are polynomials, $\deg q>1$ and $C$ is a constant, and by the \emph{$\cos$-family} the maps of the form $ae^z+be^{-z}$, where $a,b\in\mathbb{C}$. Both families are contained in $\mathcal{S}$ and include most of ``elementary'' functions  such as $e^z$, $\sin z$, $\sinh z$, $e^{p(z)}$ where $p(z)$ is a polynomial, etc.

\begin{thm}[Fast escape for the structurally finite family]
	\label{thm:structurally_finite}
	Let $f_0$ either be structurally finite or belong to the $\cos$-family, and $\lambda:\mathbb{C}\to\mathbb{C}$ be a \qc\ map equal to identity near $\infty$.
	
	If singular orbits of $f=\lambda\circ f_0$ escape $\exp$-fast and the post-singular set of $f$ is $\log$-sparse, then $f$ is Thurston equivalent to an entire function.  
\end{thm}

The $\exp$-fast escape means that there exist a constant $k>0$ such that eventually $\log \abs{f^{n+1}(s)}>k\abs{f^n(s)}$ for every singular value $s$. This is a generic mode of escape even for more general classes of entire functions (see e.g.\ \cite[Lemma~3.1]{RRRS}) and, if $f$ has positive order, it holds, in particular, for the points belonging to the \emph{fast escaping set}. We say that a set $X\subset\mathbb{C}$ is \emph{sparse} if there exists $\delta>0$ such that for any two distinct $x,y\in X$, $\abs{x-y}>\delta$. Accordingly we say that the post-singular set $P$ is \emph{$\log$-sparse}, if the set $\log (P\setminus\{0\})$ is sparse.

Most of already existing results in the direction are about some of these elementary (but not at all simple) families of functions. Markus F\"orster has shown in his PhD thesis \cite{MarkusThesis} that every ``mode'' of escape in the exponential family can be realized as the post-singular ``mode''. The approach is using pull-backs of ``spiders'' with infinitely many ``legs'' which can be interpreted as a version of Thurston's pull-back map $\sigma$. These techniques were generalized in \cite{IDTT1,IDTT2,IDTT3} for the families $p(e^z)$ where $p$ is a polynomial. We also note that in \cite{Cui}, the authors reduce the infinite-dimensional setting to the finite-dimensional using \qc\ surgery, hence apparently one cannot reproduce this approach in a neighbourhood of an essential singularity.

All these results use the explicit form of the functions in the family considered, i.e., given by a formula. In this note we get rid of this restriction by considering a general condition on the asymptotic tracts of $f_0$.

\subsection{Asymptotic area property}
Asymptotic area property which is a stronger quantitative version of the area property introduced in \cite{EpRe}. A related discussion also appears in \cite{ErLyu}.

Let $g$ be a transcendental entire function of bounded type (denoted $f\in\mathcal{B}$), i.e., having a bounded set $\SV(g)$ of singular values. For a compact $\mathcal{C}$ contained in $\mathbb{C}\setminus\SV(g)$ denote $\mathcal{E}=\mathcal{E}(g,\mathcal{C}):=f^{-1}(\mathcal{C})$ and let
$$I(\mathcal{C}):= \frac{1}{2\pi}\iint\displaylimits_{\{1\leq\abs{z}\}\bigcap\mathcal{E}}\frac{dx dy}{\abs{z}^2},$$
that is, $I(\mathcal{C})$ is the \emph{cylindrical measure} of the set $\{1\leq\abs{z}\}\bigcap\mathcal{E}$, which could be either finite or infinite. Then, according to the definition in \cite{EpRe}, $g$ has \emph{area property} if $I(\mathcal{C})<\infty$ for every $\mathcal{C}$. However, we are interested in a parametrized version of this integral.

Let $D\supset\SV(g)$ be an open set, denote $\mathcal{E}_r:=f^{-1}(\overline{\mathbb{D}}_r\setminus D)$ and consider the parametrized integral
$$I_1(\rho,D):= \frac{1}{2\pi}\iint\displaylimits_{\{\rho\leq\abs{z}\}\bigcap\mathcal{E}_\rho}\frac{dx dy}{\abs{z}^2}.$$

\begin{defn} [Asymptotic area property]
	\label{defn:as_area_property}
	We say that $f\in\mathcal{B}$ has \emph{asymptotic area property (AAP) relative to an open set $D\supset\SV(f)$} if
	$$\limsup_{\rho\to\infty}I_1(\rho,D)<\infty.$$
	
	We say that $f\in\mathcal{B}$ has \emph{AAP} if it has AAP relative to every open set $D\supset\SV(f)$.	
\end{defn}

It is self-evident that AAP implies area property and it is not difficult to show that every structurally finite function has AAP (Lemma~\ref{lmm:degeneration_structurally finite}).

We restrict to the case when $I(\rho,D)$ tends to $0$ as $\rho$ tends to $\infty$. The asymptotics of $I_1(\rho, D)$ depends on the initial choice of the domain $D$. However, in many cases (e.g.\ for finite type functions with bounded degrees of critical points) one can find a function $\chi:\mathbb{R}_+\to\mathbb{R}_+$ tending to $0$ as $\rho\to\infty$ such that for every $D$, $I_1(\rho,D)=O(\chi(\rho))$. If this is the case, we say that $\chi$ is the \emph{degeneration function} for $g$. We show in Lemma~\ref{lmm:degeneration_structurally finite} that for every structurally finite function, as well as for the $\cos$-family, we can take $\chi(\rho)=\rho^{-1/2}$, hence Theorem~\ref{thm:structurally_finite} is immediately implied by the following more general statement.
\begin{thm}[Fast escape for finite order]
	\label{thm:finite_order}
	Let $f_0\in\mathcal{S}$ be of finite order and have the degeneration function $1/\rho^\epsilon$ for some $\epsilon>0$, and $\lambda:\mathbb{C}\to\mathbb{C}$ be a \qc\ map equal to identity near $\infty$.
	
	If singular orbits of $f=\lambda\circ f_0$ escape $\exp$-fast and the post-singular set of $f$ is $\log$-sparse, then $f$ is Thurston equivalent to an entire function.
\end{thm}

The conditions on $\exp$-fast escape can be essentially weakened while the finite order can be replaced by a mild infinite order of growth (see Theorem~\ref{thm:esc_singular_orbits}). Moreover, these three parameters are interdependent. For example, the growth of $f_0$ can be chosen arbitrarily big and/or the speed of escape can be chosen slower by the price of restricting to the degeneration functions tending to $0$ faster.
 
\subsection{Infinite-dimensional Thurston theory}

Given a quasiregular function $f=\lambda\circ f_0$ as constructed above, we need to decide whether it is equivalent to an entire function. As already stated before, a general approach for answering such types of questions was developed by Thurston and Douady--Hubbard \cite{DH, HubbardBook2}: consider the pull-back map $\sigma$ acting on the corresponding \tei\ space and look for its fixed points. It is easy to see that $\sigma$ does not increase \tei\ distance, which makes plausible their existence. Thus, the question of Thurston equivalence is reduced to the study of the properties of $\sigma$.

Following the strategy of \cite{MarkusThesis, IDTT1, IDTT2, IDTT3}, in order to prove existence of a fixed point of $\sigma$, two major ingredients are required:
\begin{itemize}
	\item a $\sigma$-invariant pre-compact subset $\mathcal{I}$ of the corresponding \tei\ space,
	\item $\sigma$ should be strictly contracting on $\overline{\mathcal{I}}$.
\end{itemize}
The two conditions imply existence of a fixed point by an elementary argument.

By strict contraction we mean that $\sigma$ decreases the distances, but not necessarily with a uniform contraction factor smaller than one (hence one cannot apply the Banach Fixed Point Theorem). To address this problem we might use \cite[Lemma~4.1]{IDTT1} saying that if the $\sigma$-images of two \emph{asymptotically conformal} points (see Definition~\ref{defn:as_conformal}) are also asymptotically conformal, then $\sigma$ decreases the distances between them.

Thus, $\mathcal{I}$ must contain only asymptotically conformal points. This is one of the reasons to consider entire functions satisfying asymptotic area property. For instance, think about such $f=\lambda\circ f_0$ so that its singular orbits are very sparse near $\infty$, e.g., separated by round annuli around the origin with moduli tending to $\infty$. Then, if we apply $\sigma$ to a \tei\ equivalence class of a homeomorphism which is ``nearly identity'' in a neighbourhood of $\infty$, after pulling-back its Beltrami coefficient via $f$ and integrating it, due to \tei\--Wittich theorem~\ref{thm:teich--wittich}, we obtain an equivalence class of homeomorphism which is ``nearly identity'' on another neighbourhood of $\infty$. Therefore, the main difficulty in the construction of $\mathcal{I}$ is to arrange that the former neighbourhood of $\infty$ is contained inside of the latter. This would imply the invariance of $\mathcal{I}$ under $\sigma$.

It is natural to consider $\lambda$ as a parameter for the associated pull-back map $\sigma=\sigma(\lambda)$ acting on the corresponding \tei\ space $\mathcal{T}=\mathcal{T}(\lambda)$  and, more generally, acting on a bigger space $\hat{\mathcal{T}}=\hat{\mathcal{T}}(\lambda)$ of isotopy types of homeomorphisms (obtained from $\mathcal{T}$ if in the definition of the \tei\ equivalence classes we relax the condition for every class to contain a \qc\ map).

Let $A$ be a round annulus around the origin and $D_0$, $D_\infty$ be the bounded and unbounded components of $\mathbb{C}\setminus A$, respectively, and $\mathcal{O}$ be the union of singular orbits of $f=\lambda\circ f_0$. Restrict to $\lambda=\id$ on $A\cup D_\infty$ so that $A\cap \mathcal{O}=\emptyset$ and $\mathcal{O}\cap D_\infty$ is forward-invariant. Denote by $\mathcal{D}_0$ and $\hat{\mathcal{D}}_\infty$ the subspaces of the (resp. bigger) \tei\ spaces of $D_0$ and $D_\infty$ corresponding to the homeomorphisms equal to identity on $\partial D_0$ and $\partial D_\infty\cup\{\infty\}$, respectively. The ``gluing'' $D_0\sqcup A\sqcup D_\infty/\sim_{\partial A}$ induces the map $$\mathcal{G}:\mathcal{D}_0\times\hat{\mathcal{D}}_\infty\to\hat{\mathcal{T}},$$
contracting the associated \tei\ metric (wherever defined).

\vspace{0.5em}

The key result of this article is Theorem~\ref{thm:invariant_structure}. At this point, only its heuristic version for functions of finite order will be provided, for more details see Section~\ref{sec:invariant_structure}.

\begin{thmnonum}[Invariant set (heuristic)]
	Let $f_0$ be of finite order, have bounded degrees of criticality and $D\supset\SV(f_0)$. Let $K=K(\lambda)$ be the dilatation of $\lambda$ and $\rho$ be the radius of $D_0$. If $\mod A\gg 1$ and
	\begin{equation}
		\label{eq:heuristic}
		\left(\frac{\log(K\log\rho)}{-\log I_1(\rho,D)}\right)^{1/\abs{\mathcal{O}\cap D_0}}\ll 1,
	\end{equation}
	then there is a compact subspace $\mathcal{C}\subset\mathcal{D}_0$ and a subspace $\mathcal{P}\subset\mathcal{D}_\infty$ of ``small perturbations of identity'' such that $\mathcal{G}(\mathcal{C}\times\mathcal{P})$ is $\sigma$-invariant.
\end{thmnonum}

By the ``small perturbations of identity'' we understand the equivalence classes containing a homeomorphism which is in a small neighbourhood of identity in the $\sup$-cylindrical metric on $D_\infty$.

When $\log P$ is ``sparse enough'', $\mathcal{P}$ is a compact. Then $\sigma$ is a strictly contracting map on the compact $\mathcal{G}(\mathcal{C}\times\mathcal{P})$ hence it has a fixed point. 	

The bound (\ref{eq:heuristic}) implies that Thurston's pull-back of a $K$-\qc\ map equal to identity on $D_\infty$ is a \qc\ map with presumably very big dilatation supported on a set of very small area in $D_\infty$. In this setting, it is possible to prove a special type of Koebe-like distortion bounds and to show that a representative will be uniformly close identity on $D_\infty$.

The proof of Theorem~\ref{thm:invariant_structure} has three major ingredients depending on the type of marked points.

The marked points contained in $D_\infty$ are controlled by Koebe-like estimates from Section~\ref{sec:Koebe}. To control the behaviour of the marked points in $D_0$, we define a special structure called \emph{fat spider}. It has some similarity to the classical spider introduced in \cite{Spiders} for encoding the combinatorics of post-critically finite polynomials, but the principles of functioning are quite different. First, the ``body'' of the classical spider is the point $\infty$ while the ``body'' of a fat spider is a big disk around $\infty$ (hence ``fat''). The feet of a fat spider are finitely many marked points in the complement of the body and separated from it by an annulus of some definite modulus. Each feet is connected to the body by a homotopy class of paths (called ``legs'') in the complement of marked points and the legs are allowed to cross. To every leg we associate the maximal dilatation of a \qc\ map which maps the underlying foot to the body along the leg and (via isotopy) relative to all other feet. Given a homeomorphism equal to identity on the body so that the images of all feet are separated from the body by the same annulus we can consider the push-forward of the spider, in which its legs are just push-forwards of the legs of the initial fat spider. If we also know the maximal dilatation associated to the legs of the pushed-forward spider, we have an estimate on the maximal dilatation of the homeomorpism, see Proposition~\ref{prp:teich_metric_fat_spider_map}. On the other hand, it is not difficult to define a lift (with prolongation) of a spider leg which corresponds to the $\sigma$ map and for which it is convenient to compute the corresponding maximal dilatations. Note, that we \emph{do not} have a standard spider in the sense of \cite{Spiders}, i.e., some invariant structure formed by external rays. Instead the spider legs change on every iteration on both the left and the right hand sides the commutative diagram for $\sigma$.

Finally there is an ingredient which ``glues'' these two very different types of storing information about points in the \tei\ space. This ingredient is the property of $f_0$ which we call $(K,\delta)$-regularity of tracts. Roughly, this means the following. Consider some big $\rho>0$ and two punctured disks around $\infty$: $D:=\hat{\mathbb{C}}\setminus\mathbb{D}_\rho$ and $\hat{D}:=\hat{\mathbb{C}}\setminus\mathbb{D}_{\rho/2}$, and two logarithmic tracts $\hat{T}\supset T$ such that $f_0(\hat{T})=\hat{D}$ and $f_0(T)=D$, and assume that $T\cap\mathbb{D}_\rho\neq\emptyset$. Then we say that a pair of tracts $\hat{T}\supset T$ is $(K,\delta)$-regular if every point belonging to $\partial T\cap\mathbb{D}_\rho$ can be mapped to some point of the circle $\partial\mathbb{D}_\rho$ via a $K$-\qc\ map equal to identity outside of $\hat{T}\cap\mathbb{D}_{\rho e^\delta}$. This property allows to define a dynamically meaningful pull-back of a fat spider. It will be shown that for $f_0$ having finite order the value of $K$ can be estimated in terms of $\log\rho$. For a more detailed version with marked points see Subsection~\ref{subsec:tracts_regularity}.

\subsection{Structure of the article}
In Section~\ref{sec:standard_notions} we briefly discuss some properties of entire functions and connections to the cylindrical metric. Afterwards we provide some basic notions from the theory of \qc\ maps together with a rather lengthy list of statements we are going to use.

In Section~\ref{sec:Teich_spaces_and_Thurston_theory} we define the (extended) \tei\ space, introduce formally the $\sigma$-map and discuss its contraction properties. Finally, we prove a few semi-standard statements about different types of \qc\ representatives in the \tei\ equivalence classes.

Section~\ref{sec:AAP} is devoted to asymptotic area property. In particular, we investigate dependence of the degeneration function on $D$.

In Section~\ref{sec:Koebe} we discuss \qc\ maps with small ``total dilatation per area ratio'' and show that locally they can be approximated by identity with the quantitative bounds depending only on the ratio (Proposition~\ref{prp:distortion of identity}). In the proof of the main result (Theorem~\ref{thm:invariant_structure}) this allows to control the behaviour of marked points near $\infty$.

Section~\ref{sec:shifts_and_spiders} consists of three parts. In Subsection~\ref{subsec:shifts_properties} we introduce the homeomorphisms of a special type, called \emph{shifts}, and investigate their properties. In Subsection~\ref{subsec:tracts_regularity} we define formally the $(K,\delta)$-regularity and compute bounds for $K$ for entire functions of finite order (Proposition~\ref{prp:log_regularity_finite_order}). Finally, in Subsection~\ref{subsec:spiders} we define fat spiders and show how the maximal dilatation of a homeomorphism can be bounded using the information about the underlying fat spiders (Proposition~\ref{prp:teich_metric_fat_spider_map}).

In Section~\ref{sec:invariant_structure} we first introduce a \emph{separating structure}. This is in some sense an ``environment'' we are going to work in. It is needed in order to be able to prove Theorem~\ref{thm:invariant_structure} without fixing some particular $\lambda$ and for many different types of dynamics altogether. We construct a \emph{standard fat spider} and introduce the pull-back procedure (of standard spiders) corresponding to $\sigma$. Afterwards, we state and prove Theorem~\ref{thm:invariant_structure} and conclude the section with a rather simple Theorem~\ref{thm:fixed_point_existence} allowing to deduce existence of a fixed point for certain settings of Theorem~\ref{thm:invariant_structure}.

In Section~\ref{sec:applications} we first prove Theorem~\ref{thm:esc_singular_orbits}, and then discuss briefly perturbations of the invariant structure of Theorem~\ref{thm:invariant_structure} and the case when $f$ models a polynomial with escaping critical orbits.

\subsection{Acknowledgements}
I would like to thank Kevin Pilgrim for the series of fruitful discussions of aspects of \tei\ theory during his visit to Saarbr\"ucken. I would like to thank Nikolai Prochorov and Dierk Schleicher for their feedback, especially for the discussion related to the parameter spaces of entire functions. Also I would like to express my gratitude to Feliks Przytycki for his support during my stay at IMPAN.

\subsection{Notations and agreements}

We denote by $\mathbb{C},\hat{\mathbb{C}}$ the complex plane and the Riemann sphere, respectively. 

For $a\in\mathbb{C}$, we denote by $\mathbb{D}_r(a)$ the disk around $a$ of radius $r$. If we omit the index $r$, this means that $r=1$, if we omit the center $a$, this means that $a=0$.

By $\mathbb{D}^\infty_r$ we denote the disk $\hat{\mathbb{C}}\setminus\overline{\mathbb{D}}_r$ with center at $\infty$, and for $x\in\mathbb{R}$, $\mathbb{H}_x$ is the right half-plane $\{z:\Re z>x\}$.

For a subspace $U$ of a topological space we denote by $\partial U$ its boundary and by $U^\circ$ the interior of $U$.

For $0<r<R<\infty$, by $\mathbb{A}_{r,R}$ we denote the open round annulus between the circles $\partial\mathbb{D}_r$ and $\partial\mathbb{D}_R$. The modulus of the annulus $\mathbb{A}_{r,R}$, we define as the relation $\log(R/r)$ without the factor $1/2\pi$ in front of it. This is the convention in \cite{LehtoVirtanen} and it suits better for us for the sake of shorter formulas and simpler referencing to \cite{LehtoVirtanen}.

When we say that an isotopy is relative to some set $X$, we imply that this isotopy is constant on $\overline{X}$. In particular, when it contains the identity map, this means that all maps along the isotopy are equal to identity on $\overline{X}$. 

\section{Standard notions and definitions}
\label{sec:standard_notions}

In this section we assemble definitions and results about transcendental entire functions and \qc\ maps.

\subsection{Logarithmic coordinates and cylindrical metric}
\label{subsec:log_coordinates}

We are mainly interested in the class of entire functions of finite type (also called Speiser class or class $\mathcal{S}$) and occasionally we consider the entire functions of bounded type (also called Eremenko-Lyubich class or class $\mathcal{B}$). For the former ones, the singular set is finite, for the latter ones --- bounded.

Recall that the set of singular values is defined as the closure of the set of critical and asymptotic values. A \emph{critical value} is the image of a critical point. We say that $a\in\hat{\mathbb{C}}$ is an \emph{asymptotic value} for $f\in\mathcal{B}$ if there exists a path $\gamma:[0,1]\to\mathbb{C}$ such that $\lim_{t\to 1}\gamma(t)\to\infty$ and $\lim_{t\to 1}f\left(\gamma(t)\right)\to a$. In particular, $\infty$ is always a singular value for $f$. However, when speaking about singular values, we always mean a \emph{finite} singular value and denote their set by $\SV(f)$.

Following \cite{ErLyu}, we say that an entire function $g$ belongs to the \emph{parameter space} of $f_0$ if there exist \qc\ homeomorphisms $\varphi_1, \psi_1:\mathbb{C}\to\mathbb{C}$ such that $g\circ\psi_1=\varphi_1\circ f_0$. It is easy to see that for a finite type function $f_0$, if $f$ is Thurston equivalent to an entire function $g$, then $g$ belongs to the parameter space of $f_0$.

By the \emph{post-singular set} $P$ of $f\in\mathcal{S}$ we understand the closure of the union of forward orbits of singular values (including the singular value itself). Note that very often we need to distinguish between the union of singular orbits (without taking the closure) and $P$. In this case we denote this union by $\mathcal{O}$. 

For $f\in\mathcal{B}$, the \emph{logarithmic coordinates} are introduced as follows. Let $R$ be big enough, so that $\SV(f)\cup\{f(0)\}\subset\mathbb{D}_R$, and let $\mathcal{T}:=f^{-1}(\mathbb{C}\setminus\overline{\mathbb{D}}_R)$. Then $\mathcal{T}$ is a union of unbounded simply-connected connected components called (logarithmic) \emph{tracts} of $f$. Then the restriction $f|_{\mathcal{T}}$ can be lifted via the exponential map. The derived function is called \emph{logarithmic transform} of $f$ and is denoted by $F$.
\begin{center}
	\begin{tikzcd}
		\tilde{\mathcal{T}}\arrow[r, "{F}"] \arrow[d, swap, "\exp"]	& \mathbb{H}_{\log R} \arrow[d, "\exp"] \\
		\mathcal{T} \arrow[r, "{f}"] & \mathbb{C}\setminus\overline{\mathbb{D}}_R
	\end{tikzcd}
\end{center}
\vspace{0.5cm}

$F$ is defined on the $2\pi i$-periodic set $\tilde{\mathcal{T}}$ of pre-images of tracts (also called \emph{tracts}) under the exponential. The restriction of $F$ to each tract is a conformal map onto the right half-plane $\mathbb{H}_{\log R}$.

The logarithmic coordinates are well suited for orbits staying ``near $\infty$'', in particular, for the escaping points. The most important feature of these coordinates is that for big enough $R$, $F$ is expanding, and in a quite strong way \cite[Lemma 1]{ErLyu}:
$$\abs{F'(z)}\geq\frac{1}{4\pi}\left(\Re F(z)-\log R\right).$$
We will not require the explicit inequality in this article, but we will occasionally use the strong expansivity of $F$.

The \emph{cylindrical area} is defined on $\mathbb{C}\setminus\{0\}$ by the area element $dxdy/\abs{z}^2$. For $z,w\in\mathbb{C}\setminus\{0\}$, we will denote by $d_{\cyl}(z,w)$ the distance between points in the cylindrical metric and call it \emph{cylindrical distance}. Note that its pull-back under the exponential coincides with the Euclidean metric --- we will use this property regularly throughout the article.

\subsection{Quasiconformal maps}
The most standard references are \cite{Ahlfors,BrannerFagella,LehtoVirtanen}, though the latter one will be used most intensively in this article.

\begin{defn}[Quadrilateral]
	A \emph{quadrilateral} $Q(z_1,z_2,z_3,z_4)$ is a Jordan domain $Q$ together with a sequence $z_1,z_2,z_3,z_4$ of boundary points called vertices of the quadrilateral. The order of vertices agrees with the positive orientation with respect to $Q$. Arcs $z_1 z_2$ and $z_3 z_4$ are called $a$-sides, arcs $z_2 z_3$ and $z_4 z_1$ are called $b$-sides.
\end{defn}

Every such quadrilateral $Q$ is conformally equivalent to the unique canonical rectangle with the length of $b$-sides equal to 1. For a quadrilateral $Q$, the length of the $a$-sides of the canonical rectangle is called \emph{modulus} of $Q$ and denoted $\mod Q$.

Analogously, every annulus $A$ is conformally equivalent to a unique round annulus $\mathbb{A}_{1,R}$ for some $R>1$. Then the modulus of $A$ is defined as $\mod A=\log R$. Note that it is more colloquial to have a factor $1/2\pi$ in front of the logarithm. However, we will follow the convention in \cite{LehtoVirtanen} without it.

A maximal dilatation can be defined both in terms of the moduli of quadrilaterals and of the moduli of annuli. The defined objects coincide.

\begin{defn}[Maximal dilatation]
	Let $U$ and $V$ be planar domains and $\psi:U\to V$ be an orientation-preserving homeomorphism. The \emph{maximal dilatation} of $\psi$ is called the number
	
	\begin{center}
		$K_\psi=\sup_{\overline{Q}\subset U}\frac{\mod \psi(Q)}{\mod Q}$,\\
	\end{center}
	where the supremum is taken over all quadrilaterals $Q$ (resp. annuli) contained in $U$ together with its boundary.
\end{defn}

Using $K_\psi$ we can define \qc\ maps.

\begin{defn}[Quasiconformal map]
	An orientation-preserving homeomorphism $\psi$ of a plane domain $U$ is called quasiconformal if its maximal dilatation $K_\psi$ is finite. If $K_\psi\leq K<\infty$, then $\psi$ is called $K$-quasiconformal.
\end{defn}

The inverse of a $K$-quasiconformal map is also $K$-quasiconformal, while the composition of a $K_1$-quasiconformal and $K_2$-quasiconformal map is $K_1 K_2$-quasiconformal.

Quasinconformal maps can also be defined analytically.

\begin{defn}[Quasiconformal map]
	A homeomorphism $\psi$ of a plane domain $U$ is quasiconformal if there exists $k<1$ such that
	
	\begin{enumerate}
		\item $\psi$ has locally integrable, distributional derivatives $\psi_z$ and $\psi_{\overline{z}}$ on $U$, and
		\item $\abs{\psi_{\overline{z}}} \leq k\abs{\psi_z}$ almost everywhere.
	\end{enumerate}
	
	Such $\psi$ is called $K$-quasiconformal, where $K=\frac{1+k}{1-k}$.
\end{defn}

Each quasiconformal map is determined up to a post-composition by a conformal map by its Beltrami coefficient.

\begin{defn}[Beltrami coefficient]
	The function $\mu_\psi(z)=\psi_{\overline{z}}(z)/{\psi_z (z)}$ is called the \emph{Beltrami coefficient} of $\psi$. It is defined almost everywhere on $U$.
\end{defn}

Providing the Beltrami coefficient is almost the same as providing a quasiconformal map. Consider the Beltrami equation
$$\psi_{\overline{z}}(z)=\mu (z) \psi_z (z)$$
where the partial derivatives $\psi_z (z)$ and $\psi_{\overline{z}}(z)$ are defined in the sense of distributions and are locally integrable.

\begin{thm}[Measurable Riemann Mapping Theorem \cite{GardinerLakic}]
	The Beltrami equation gives a one-to-one correspondence between the set of quasiconformal homeomorphisms of $\hat{\mathbb{C}}$ that fix the points $0,1$ and $\infty$ and the set of measurable complex-valued functions $\mu$ on $\hat{\mathbb{C}}$ for which $\lvert\lvert\mu\rvert\rvert_{L^\infty}<1$.
\end{thm}

We finish this subsection by providing a list of somewhat more technical statements about \qc\ maps.

\begin{thm}[Theorem 2.1, \cite{McMullen_book}]
	\label{thm:essential_round_annulus}
	Any annulus $A\subset\mathbb{C}$ of sufficiently large modulus contains an essential (i.e., separating the boundary components of $A$) round annulus $B$ with $\mod A = \mod B + O(1)$.
\end{thm}

\begin{lmm}\cite[Section I.4.4, Rengel's inequality]{LehtoVirtanen}
	\label{lmm:Rengel}
	Let $Q\subset\mathbb{C}$ be a quadrilateral with (Euclidean) area $m(Q)$ and $s_a(Q), s_b(Q)$ be Euclidean distances between its a-sides and b-sides respectively (measured along paths inside of $Q$). Then
	$$\frac{(s_b(Q))^2}{m(Q)}\leq \mod Q\leq \frac{m(Q)}{(s_a(Q))^2}.$$
	The inequality in each case is possible if and only if $Q$ is a rectangle.  
\end{lmm}

\begin{lmm}\cite[Chapter II, inequality (9.1)]{LehtoVirtanen}
	\label{lmm:max_over_min}
	Let $\varphi:\mathbb{D}\to U\subset\mathbb{C}$ be a $K$-\qc\ map such that $\varphi(0)=0$ and for some $\alpha,r>0$, $\varphi(\partial\mathbb{D}_\alpha)\subset\mathbb{D}_r\subset U$. Then
	$$\max\displaylimits_{z\in\partial\mathbb{D}_\alpha}\abs{\varphi(z)}\leq e^{\pi K}\min\displaylimits_{z\in\partial\mathbb{D}_\alpha}\abs{\varphi(z)}.$$
\end{lmm}

For a \qc\ map $\varphi$, let $D_\varphi(z):=\frac{1+\abs{\mu_\varphi(z)}}{1-\abs{\mu_\varphi(z)}}$.

\begin{lmm}
	\label{lmm:ineq_mod_rectangle}
	Let $Q$ be a quadrilateral such that its $a$-sides lie on different sides of a horizontal strip of height $h>0$ and let $\varphi$ be a \qc\ map of $Q$. Then
	$$\mod \varphi(Q)\leq\frac{1}{h^2}\iint\displaylimits_Q D_\varphi(z) dxdy.$$
\end{lmm}
\begin{proof}
	The proof is the same as for the analogous inequality for rectangles in \cite[Chapter V, Section 6.3]{LehtoVirtanen}. Following the notation of the book, we have to assign $\rho(z)=1/h$ when $z$ belongs to the intersection of $Q$ with the strip, and $\rho(z)=0$ otherwise.
\end{proof}

\begin{lmm}\cite[Chapter V, inequality (6.9)]{LehtoVirtanen}
	\label{lmm:ineq_mod_diff}
	Let $A$ be a round annulus around $0$ and $\varphi$ be a \qc\ map of $A$. Then
	$$\abs{\mod\varphi(A)-\mod A}\leq\frac{1}{2\pi}\iint\displaylimits_A \frac{D_\varphi(z)-1}{\abs{z}^2}dxdy.$$
\end{lmm}

The following result is a part of the \tei\--Wittich theorem. In this article, we follow the exposition of \cite{LehtoVirtanen}. An alternative proof of the theorem can be found in \cite{Mitsu_teich_wittich}. Another reference for a similar type of results is \cite{Bishop_thin}.

\begin{thm}\cite[Satz 6.1, \tei--Wittich theorem]{LehtoVirtanen}
	\label{thm:teich--wittich}
	Let $\varphi:\mathbb{C}\to\mathbb{C}$ be a \qc\ map such that $\varphi(0)=0$ and
	$$\frac{1}{2\pi}\iint\displaylimits_{\abs{z}<1} \frac{D_\varphi(z)-1}{\abs{z}^2}dxdy<\infty.$$
	Then $\varphi$ is complex differentiable (conformal) at $0$.
\end{thm}

\section{\tei\ spaces and Thurston's theory}
\label{sec:Teich_spaces_and_Thurston_theory}
\subsection{\tei\ spaces}
\label{subsec:teich_spaces}
Good references for the quasiconformal \tei\ theory are \cite{GardinerLakic, HubbardBook1}.

\begin{defn}[\tei\ space of $\mathbb{C}\setminus P$]
	\label{defn:tei_space}
	For a set $P\subset\mathbb{C}$, the \emph{\tei\ space} $\mathcal{T}_P$ of $\mathbb{C}$ with the \emph{marked set} $P$ is the set of \emph{quasiconformal} homeomorphisms of $\mathbb{C}\setminus P$ modulo post-composition with an affine map and isotopy relative to $P$ (or, equivalently, relative to $\overline{P}$).
	
	By $\hat{\mathcal{T}}_P$ we denote the set of \emph{topological} homeomorphisms of $\mathbb{C}$ modulo post-composition with an affine map and isotopy relative to $P$.
\end{defn}

\begin{remark}
	A more standard definition of the \tei\ space on a Riemann surface involves isotopy relative the ideal boundary rather than the topological boundary. For planar domains the two definitions are equivalent \cite{GardinerLakic}.
\end{remark}

Every \tei\ space is equipped with the \tei\ metric with respect to which $\mathcal{T}_P$ is a complete metric space.

\begin{defn}[\tei\ distance]
	Let $[\varphi_0],[\varphi_1]\in\mathcal{T}_P$. The \tei\ distance $d_\mathcal{T}([\varphi_0],[\varphi_1])$ is defined as	
	$$\inf\limits_{\psi\in [\varphi_1\circ (\varphi_0)^{-1}]} \log K_\psi.$$
\end{defn}

Clearly, $\mathcal{T}_P$ is contained in $\hat{\mathcal{T}}_P$ and consists exactly of the equivalence classes containing a \qc\ map. If $P$ is finite, then $\hat{\mathcal{T}}_P=\mathcal{T}_P$.

The points belonging to the set $\overline{P}$ are called \emph{marked points}. Since an isolated point is a removable singularity for a \qc\ map, our setting agrees with the more colloquial one when one considers the Riemann sphere $\hat{\mathbb{C}}$ instead of $\mathbb{C}$: the formal difference lies in either ``forgetting'' about $\infty$ (as we do) or making it a marked point (not having any dynamical meaning later on).
 
For $P'\subset P$ and $[\psi]\in\hat{\mathcal{T}}_P$, we denote by $[\psi]_{P'}$ the projection of $[\psi]$ to $\hat{\mathcal{T}}_{P'}$, i.e.\ the \tei\ equivalence class $[\psi]_{P'}$ is defined as the image of the class $[\psi]$ under the \emph{forgetful} map which ``forgets'' the marked points $\overline{P}\setminus\overline{P'}$.

\subsection{Setup of Thurston's iteration}

\label{subsec:iteration_setup}
The construction is described in \cite{IDTT1,IDTT2} for a more specific class of entire function.

Let $f_0$ be a transcendental entire function of bounded type, $\lambda:\mathbb{C}\to\mathbb{C}$ be a \qc\ map and $f=\lambda\circ f_0$. Further, let $P\subset\mathbb{C}$ be a forward invariant set containing $\SV(f):=\lambda(\SV(f_0))$. The most important example for us is when $f_0$ is of finite type and all singular values of $f$ escape. In this case the union of the orbits of singular values can be chosen as $P$.

The quasiregular map $f$ defines Thurston's pull-back map
$$\sigma:[\varphi]\in\mathcal{T}_P\mapsto[\tilde{\varphi}]\in\mathcal{T}_P,$$
acting on the \tei\ space $\mathcal{T}_P$, via the following diagram: 
\begin{center}
	\begin{tikzcd}
		\mathbb{C},P \arrow[r, "{\tilde{\varphi}}"] \arrow[d, "f"]	& \mathbb{C},\tilde{\varphi}(P) \arrow[d, "g"] \\
		\mathbb{C},P \arrow[r, "{\varphi}"] & \mathbb{C},\varphi(P)
	\end{tikzcd}
\end{center}
\vspace{0.5cm}

More precisely, due to Measurable Riemann Mapping Theorem, applied to the pull-back of the Beltrami coefficient of $\varphi\circ\lambda$ via $f_0$, for every \qc\ map $\varphi:\mathbb{C}\to\mathbb{C}$ there is another \qc\ map $\tilde{\varphi}:\mathbb{C}\to\mathbb{C}$ such that $g=\varphi\circ f\circ\tilde{\varphi}^{-1}$ is an entire function. Thus, define $\sigma[\varphi]:=[\tilde{\varphi}]$. The standard lifting argument shows that $\sigma$ is well defined.

Note that according to our definition of $\sigma$ we do not consider $\sigma[\varphi]$ for an arbitrary topological homeomorphism $\varphi$ because, generally speaking, if $P$ is infinite, the equivalence class $[\varphi]$ might not contain any single \qc\ map. If this is the case, there is no Beltrami coefficient to pull-back and integrate. And even if $\sigma$ is defined, $[\varphi]$ and $[\tilde{\varphi}]$ might belong to different \tei\ spaces (see \cite[Section 3.3]{IDTT1}) which makes it impossible to use the contracting properties of $\sigma$. However, this setup will still be useful if we restrict to finite type entire functions $f_0$. Then the domain of definition of $\sigma$ can be extended to $\hat{\mathcal{T}}_P$ as follows. Let $[\psi]\in\hat{\mathcal{T}}_P$ for some \emph{topological} homeomorphism  $\psi:\mathbb{C}\to\mathbb{C}$. There is a \qc\ map $\psi'\in[\psi]_{\SV(f)}\in\hat{\mathcal{T}}_{\SV(f)}$ such that $\psi'|_{\SV(f)}=\psi|_{\SV(f)}$, i.e., a \qc\ representative of the projection of $[\psi]\in\hat{\mathcal{T}}_P$ to $[\psi]_{\SV(f)}\in\hat{\mathcal{T}}_{\SV(f)}$. Moreover, there exists an isotopy $\psi_t:\mathbb{C}\to\mathbb{C}, t\in[0,1]$ between $\psi$ and $\psi'$ relative to $\SV(f)$. Thus, to obtain $\sigma[\psi]$, we first choose some $\tilde{\psi}'\in\sigma[\psi']_{\SV(f)}$ in the usual way and then lift the isotopy $\psi_t$ starting at $\tilde{\psi}'$. The terminal point will be $\tilde{\psi}$. The usual lifting argument shows that such prolongation of $\sigma$ is well defined and there is a map 
$$\sigma:[\psi]\in\hat{\mathcal{T}}_P\mapsto[\tilde{\psi}]\in\hat{\mathcal{T}}_P.$$

Analogously, if $P', P\subset\mathbb{C}$ are two sets (not necessarily forward invariant) such that $\SV(f)\subset P$ and $f(P')\subset P$, one can interpret $\sigma$ as a map
$$\sigma:[\psi]\in\hat{\mathcal{T}}_P\mapsto[\tilde{\psi}]\in\hat{\mathcal{T}}_{P'}.$$
It will be clear from the context which exactly version of $\sigma$ is under consideration.

The definition of the Thurston equivalence is the same as in the finite-dimensional setting.
\begin{defn}[Thurston equivalence]
	\label{defn:Thurston_equivalence}
	We say that $f$ is \emph{Thurston equivalent} to the entire map $g$ if there exist two homeomorphisms $\varphi,\psi:\mathbb{C}\to\mathbb{C}$ such that
	\begin{enumerate}
		\item $\varphi=\psi$ on $P$,
		\item the following diagram commutes
		
		\begin{tikzcd}
			\mathbb{C},P \arrow[r, "{\psi}"] \arrow[d, "f"]	& \mathbb{C},\psi(P) \arrow[d, "g"] \\
			\mathbb{C},P \arrow[r, "{\varphi}"] & \mathbb{C},\varphi(P)
		\end{tikzcd}
		\item $\varphi$ is isotopic to $\psi$ relative to $P$.
	\end{enumerate}
\end{defn}

Fixed points of $\sigma$ correspond to the entire functions which are Thurston equivalent to $f$.

\begin{lmm}\cite[Proposition 2.3]{DH}
	The quasiregular function $f$ is Thurston equivalent to an entire function if and only if $\sigma|_{\hat{\mathcal{T}}_P}$ has a fixed point.
\end{lmm}

However, in order to apply the contraction properties and deduce existence of a fixed point we need $\sigma$ to act on some \tei\ space (which is not always the case if $P$ is infinite). Very often one can deal with this obstacle by switching to a certain $\sigma$-invariant subset of $\hat{\mathcal{T}}_P$.

\subsection{Strict contraction}
\label{subsec:strict_contraction}

It is easy to see that if $\sigma$ is contracting on $\mathcal{T}_P$, i.e., for every pair of distinct $[\varphi],[\psi]\in\mathcal{T}_P$,
$$d_\mathcal{T}(\sigma[\varphi],\sigma[\psi])\leq d_\mathcal{T}([\varphi],[\psi]),$$
where $d_\mathcal{T}$ is the \tei\ metric. For proof, one needs to pull-back $\varphi\circ\psi^{-1}$ via the entire map appearing on the right hand side of the Thurston diagram for $\psi$. The maximal dilatation will not increase.

However, this is not enough for deducing existence of a fixed point of $\sigma$ and a stronger condition is required. In fact, $\sigma$ can be strictly contracting on a particular subset of $\mathcal{T}_P$.

\begin{defn}[Asymptotically conformal points \cite{GardinerLakic}]
	\label{defn:as_conformal}
	A point $[\varphi]\in\mathcal{T}_P$ is called \emph{asymptotically conformal} if for every $\varepsilon>0$ there is a compact set $\mathcal{C}\subset\mathbb{C}\setminus \overline{P}$ and a \qc\ representative $\psi\in[\varphi]$ such that $\abs{\mu_\psi}<\varepsilon$ almost everywhere on $(\mathbb{C}\setminus \overline{P})\setminus \mathcal{C}$.
\end{defn}

Consider the quasiregular function $f=\lambda\circ f_0$ constructed in the previous subsection. Then the pull-back map $\sigma$ associated to $f$ tends to decrease distances between asymptotically conformal points if their $\sigma$-images are asymptotically conformal.

\begin{lmm}[Strict contraction of $\sigma$]
	\label{lmm:strict_contraction}
	Assume that every singular value of $f=\lambda\circ f_0$ is either escaping or strictly pre-periodic. Let $\mathcal{A}\subset\mathcal{T}_P$ be $\sigma$-invariant set containing only asymptotically conformal points.
	
	Then some iterate $\sigma^n, n>0$ is strictly contracting on $\mathcal{A}$, or, equivalently, if $[\varphi],[\psi]\in\mathcal{A}$, then
	$$d_\mathcal{T}(\sigma^n[\varphi],\sigma^n[\psi])<d_\mathcal{T}([\varphi],[\psi]).$$
\end{lmm}
\begin{proof}[Sketch of proof]
	The lemma is a simple upgrade of \cite[Lemma 4.1]{IDTT1} allowing orbits to merge. We show now how its proof should be modified by using the notation from \cite[Lemma 4.1]{IDTT1}. Recall that by the Great Picard Theorem an entire function assumes every complex value with at most one exception and the omitted value is necessarily an asymptotic value. Therefore, if $\sigma$ does not contract strictly the distance between $[\varphi]$ and $[\psi]$, the quadratic differential $q_0$ cannot have poles associated to the marked points which are not singular values (in particular to the marked points corresponding to cycles). This means that $q_0$ has finitely many poles and its pull-back coincides with $q_1$, but the indices of the marked points with an associated pole are decreased by one. This procedure can be repeated at most finitely many times, say $m$ (depending only on the orbit portrait): otherwise we obtain an integrable quadratic differential without poles, hence equal to $0$. Thus, we can take $n=m$.
\end{proof}

\subsection{Representatives of \tei\ equivalence classes}
\label{subsec:teich_representatives}
We prove a few rather technical statements about existence of a suitable for us representative in the \tei\ equivalence classes. Similar statements are represented in the literature, but often in a slightly different form (e.g., without explicit bounds for the dilatation). Therefore, we provide them in the form suitable for the ad hoc application.

The lemma below, though elementary, will be used in multiple places throughout the article after trivial modifications (such as for disks of different radii).

\begin{lmm}
	\label{lmm:conformal_neighbourhood_expand}
	Let $\psi:\overline{\mathbb{D}}\to\overline{\mathbb{D}}$ be a $K$-\qc\ map such that $\psi(0)=0$ and $\psi|_{\mathbb{D}_r}$ is conformal for some $r\in (0,1)$. If $1>\rho>r$, then $\psi$ can be isotoped relative to $\partial\mathbb{D}\cup\{0\}$ to a $K'$-\qc\ map $\varphi$ so that it is conformal on $\mathbb{D}_\rho$ and $K'=K\log{r}/\log{\rho}$.
	
	Analogously, if $\psi|_{\mathbb{D}_r}=\id$ for some $r\in (0,1)$ and $1>\rho>r$, then $\psi$ can be isotoped relative to $\partial\mathbb{D}\cup\{0\}$ to a $K'$-\qc\ map $\varphi$ so that $\varphi|_{\mathbb{D}_\rho}=\id$ and $K'=K(\log{r}/\log{\rho})^2$.
\end{lmm}
\begin{proof}
	Construct an auxiliary isotopy $\psi_t, t\in [r,\rho]$ of $\overline{\mathbb{D}}$ as follows. On the annulus $A=\{z: 1>\abs{z}>\rho\}$ define $\psi_t(z):=z\abs{z}^{\frac{\log{t}}{\log{\rho}}-1}$, while on $\overline{\mathbb{D}}_\rho$ let $\psi_t(z):=z t/\rho$. Then the composition $\psi\circ\psi_t, t\in [r,\rho]$ is the desired isotopy between $\psi=\psi\circ\psi_\rho$ and a $K'$-\qc\ map $\psi\circ\psi_r$, which is conformal on $\mathbb{D}_\rho$.
	
	For the case $\psi|_{\mathbb{D}_r}=\id$ we might consider the isotopy $\psi_t^{-1}\circ\psi\circ\psi_t, t\in [r,\rho]$ instead.
\end{proof}

Next two lemmas are providing bounds on the maximal dilatation if we isotope a \qc\ map in a neighbourhood of a point in order to make it either conformal or equal to identity there. It is clear that the constant $1/2$ can be replaced by any other constant smaller than $1$.

\begin{lmm}[Conformality near puncture]
	\label{lmm:conformal_neighbourhood}
	Let $\psi:\overline{\mathbb{D}}\to\overline{\mathbb{D}}$ be a $K$-\qc\ map such that $\psi(0)=0$. Then $\psi$ is isotopic relative to $\partial\mathbb{D}\cup\{0\}$ to a $K'$-\qc\ map $\varphi$ such that $\varphi|_{\mathbb{D}_{1/2}}$ is conformal and $K'=O(K^2)$. 
\end{lmm}
\begin{proof}
	Let $r\in (0,1)$: its particular value will be chosen later. Consider a \qc\ map $\psi_1:\mathbb{D}\to\mathbb{D}$ fixing the origin such that its Beltrami coefficient is equal to $0$ in $\mathbb{D}_r$ and equal to the Beltrami coefficient of $\psi$ in $A_r:=\{z:1>\abs{z}>r\}$. Then the map $\psi_2:=\psi\circ\psi_1^{-1}$ is $K$-\qc\ and conformal on $\psi_1(A_r)$. 
	
	Now, we construct a suitable isotopy of $\psi_2$, which is conformal on $\psi_1(\mathbb{D}_r)$. Note first that since the modulus of $\psi_1(A_r)$ is at least $-\log r/K$, the image $\psi_1(\mathbb{D}_r)$ is contained in $\mathbb{D}_{r_2}$ where $r_2=\sqrt{-8K/\log r}$. Indeed, if $\epsilon=\sup_{z\in\mathbb{D}_r}\abs{\psi_1(z)}$, then $\mod \psi_1(A_r)\leq 2\pi^2/(2\epsilon)^2<8/\epsilon^2$ (use, for example, \cite[Formula I.6.1]{LehtoVirtanen} with the Euclidean metric $\rho$), hence $\epsilon<\sqrt{-8K/\log r}$.
	
	Assume that $r_2$ is chosen in such a way that $\mathbb{D}_{r_2}$ has hyperbolic diameter (in $\mathbb{D}$) not exceeding $\delta=1/2$, i.e., $r_2<(e^{\delta/2}-1)/(e^{\delta/2}+1)$. Then as in the proof of \cite[Lemma 4.2]{IDTT1} we see that the quasisymmetry constant of the function $\psi_2^{-1}|_{\partial\mathbb{D}}$ is bounded independently of $K$ (by $\lambda(1/(1-\delta))=\lambda(2)$ following the notation of \cite[Lemma 4.2]{IDTT1}). It follows from \cite[Proposition 2.28]{BrannerFagella} and the Alexander trick that $\psi_2$ can be isotoped relative $\partial\mathbb{D}\cup\{0\}$ to a \qc\ map $\psi'_2$ such that $\psi'_2|_{\mathbb{D}_{r_2}}=\id$ and its maximal dilatation is universally bounded.
	
	Restrict to $r_2<1/8$ or, equivalently, $r<e^{-2^9 K}$. Then the composition $\psi'_2\circ\psi_1$ is conformal on $\mathbb{D}_r$ and has maximal dilatation not exceeding $K_0 K$, for some universal constant $K_0\geq 1$.
	
	Note that $1/2>e^{-2^9 K}$, but using Lemma~\ref{lmm:conformal_neighbourhood_expand} we can additionally isotope $\psi'_2\circ\psi_1$ to a map having the maximal dilatation at most $K_0 K \log{e^{-2^{9}K}}/\log(1/2)$.  
\end{proof}

\begin{lmm}[Identity near puncture]
	\label{lmm:id_near_puncture}
	Let $0\in U,V\subset\hat{\mathbb{C}}$ be two Riemann domains, $r,R>0$ be maximal radii so that $\mathbb{D}_r\subset U,\; \mathbb{D}_R\subset V$ and $\psi: U\to V$ be a $K$-\qc\ map such that $\psi(0)=0$. Then $\psi$ is isotopic relative to $\partial U\cup\{0\}$ to a $K'$-\qc\ map $\varphi$ such that $\varphi|_{\mathbb{D}_{\min(r, R)/2}}=\id$ and $K'=O(K^4)\left(1+\abs{\log R/r}\right)$.
\end{lmm}
\begin{proof}[Sketch of proof]
	First, let us prove the lemma for $U=V=\mathbb{D}$.
	
	After applying Lemma~\ref{lmm:conformal_neighbourhood_expand} and Lemma~\ref{lmm:conformal_neighbourhood} we might assume that $\psi$ is conformal on $\mathbb{D}_{3/4}$ and has maximal dilatation $K_1=O(K^2)$. Let $A:=\mathbb{D}\setminus\overline{\mathbb{D}_{3/4}}$. From the Gr\"otzsch inequality, $\max_{z\in \partial\mathbb{D}_{3/4}}\abs{\psi(z)}\geq \left(3/4\right)^{K_1}$. Applying also Lemma~\ref{lmm:max_over_min} we obtain
	$$\min_{z\in \partial\mathbb{D}_{3/4}}\abs{\psi(z)}\geq e^{-\pi K_1}\max_{z\in \partial\mathbb{D}_{3/4}}\abs{\psi(z)}\geq e^{-\pi K_1}\left(3/4\right)^{K_1}=e^{-O(K^2)}.$$
	That is, $\psi(\mathbb{D}_{3/4})$ contains a round disk around $0$ of a radius $e^{-O(K^2)}$. After post-composing $\psi$ with the isotopy from Lemma~\ref{lmm:conformal_neighbourhood_expand}, we might assume additionally that $\psi(\mathbb{D}_{3/4})$ contains $\mathbb{D}_{3/4}$ and has maximal dilatation $K_2=O(K^2) K_1=O(K^4)$. The Koebe $1/4$-Theorem implies that $\abs{\psi'(0)}$ is universally bounded from above and from below. The proof can be concluded by interpolation as in \cite[Proposition 2.28]{BrannerFagella} together with the standard Koebe distortion argument applied to the conformal map $\psi|_{\mathbb{D}_{3/4}}$ (note that the ``quasisymmetric'' constants are universally bounded on $\mathbb{D}_{r}$ if $r<3/4$ is close enough to $0$).
	
	More generally, if $U,V$ are such that $r=R=1$, pre- and post-composition of $\psi$ with the respective Riemann maps fixing $0$ reduces the problem to the solved case above. On the other hand, for the Riemann maps the derivatives at $0$ are universally bounded from above and below, hence one can isotope them to identity on $\mathbb{D}_{1/2}$ with the universal bound on maximal dilatation. Thus, as in the case of the unit disks, $K'=O(K^4)$.
	
	Finally, let $U,V$ be arbitrary Riemann domains. Consider the map $z\mapsto\psi(rz)/R$ defined on the rescaled domain. Applying the conclusion of the previous paragraph we see that $\psi$ can be isotoped relative to $\partial U\cup \{0\}$ to a map which is equal to $z\mapsto zR/r$ on $\mathbb{D}_{r/2}$ and has maximal dilatation $O(K^4)$. This new map can be isotoped to a map $\varphi$ equal to identity on $\mathbb{D}_{\min(r, R)/4}$ and having maximal dilatation $O(K^4)\left(1+\abs{\log R/r}\right)$. Using Lemma~\ref{lmm:conformal_neighbourhood_expand} and increasing the $O(.)$ bounds we might extend the domain on which $\varphi=\id$ to $\mathbb{D}_{\min(r, R)/2}$.
\end{proof}

Next two statements are giving an answer to the following question. Given a Riemann domain with two marked points in it, how big is the maximal dilatation induced by moving one point into another by an isotopy relative to the boundary of the domain?

\begin{lmm}
	\label{lmm:move_inside_annulus_bounds}
	Let $A\subset \mathbb{C}$ be an annulus and let points $x,y$ be contained in the bounded component of $\mathbb{C}\setminus A$. Then there exists a $K$-\qc\ map $\psi:\mathbb{C}\to\mathbb{C}$ equal to identity on the unbounded component of $\mathbb{C}\setminus A$ such that $\psi(x)=y$ and $K=O(1+1/(\mod A)^2)$.
\end{lmm}
\begin{proof}
	Let $D$ be the union of $A$ with the bounded component of $\mathbb{C}\setminus A$. Without loss of generality we might assume that $D$ is the unit disk and $x=0, y=r$ for some $r\in[0,1)$. The annulus of the biggest modulus separating $0$ and $r$ from $\partial D$ is the Gr\"otzsch extremal domain $D\setminus[0,r]$ (see \cite[Section II.1]{LehtoVirtanen}). Therefore, we might assume that $A=D\setminus[0,r]$. Next, after changing coordinates via applying a M\"obius transformation we assume that $A=D\setminus [-r_1,r_1]$ for some $r_1<r$. From the central simmetry of $A$ one sees that it is enough to apply a half twist exchanging $-r_1$ and $r_1$ and leaving the interval $[-r_1,r_1]$ invariant, in order to provide a \qc\ map exchanging $-r_1$ and $r_1$. Most easily this can be done for the round annulus of modulus equal to $\mod A$: direct computation shows that the induced maximal dilatation will be equal to $O(1+1/(\mod A)^2)$.
\end{proof}

The lemma above can be restated in the form which is more convenient for us.

\begin{cor}
	\label{cor:moving_inside_round_disk_bounds}
	Let $\delta\in (0,1/2)$ and $x\in\overline{\mathbb{D}}_{1-\delta}$. Then $x$ can be mapped to any other point of $\overline{\mathbb{D}}_{1-\delta}$ by a $K$-\qc\ map equal to identity on $\partial\mathbb{D}$ and with $K=O(\log^2\delta)$.
\end{cor}
\begin{proof}
	We can restrict to the case when $\delta\to 0$. As in the lemma above, if $x,y\in\mathbb{D}_{1-\delta}$, then the annulus of maximal modulus in $\mathbb{D}$, separating them from $\mathbb{D}_1^\infty$ is the complement in $\mathbb{D}$ of the hyperbolic geodesic segment joining $x$ to $y$. Thus, this modulus will be the smallest if $\abs{x}=1-\delta$ and $y=-x$, i.e., the hyperbolic distance between them is equal to $2\log\left(\frac{1+\abs{x}}{1-\abs{x}}\right)=2\log\left(\frac{2-\delta}{\delta}\right)$. After a holomorphic change of coordinates we might assume that $y=0$ and $\abs{x}=1-\delta'$ where $\delta'=\left(1/2+O(\delta)\right)\delta^2$.
	
	The annulus of the largest modulus separating $x$ and $0$ from $\partial\mathbb{D}$ is the Gr\"otzsch extremal domain $\mathbb{D}\setminus [0,x]$ having modulus $\mu(\abs{x})=\mu(1-\delta')$ (see \cite[Section II.1]{LehtoVirtanen} for the definition of the function $\mu$). From \cite[Section 2.1, Equation (2.7)]{LehtoVirtanen} and \cite[Section 2.1, Equation (2.11)]{LehtoVirtanen},
	$\mu(1-\delta')\log\delta'$ converges to a negative constant as $\delta'\to 0$. Using the estimate of Lemma~\ref{lmm:move_inside_annulus_bounds}, we obtain the required bound.
\end{proof}

We finish this subsection by a short computation needed to bound the maximal number of twists happening under a $K$-\qc\ automorphism of the thrice punctured sphere. The bound is quite rough, but it will be sufficient for our needs.

\begin{lmm}[Twist angle in thrice punctured sphere]
	\label{lmm:twist_in_3_punctured_sphere}
	Let $p\in\mathbb{D}_{2}^\infty\setminus\{\infty\}$ and $\psi:\mathbb{C}\to\mathbb{C}$ be a $K$-\qc\ homeomorphism isotopic relative to $\{0,1,p\}$ to an $n$-twist of the annulus $\mathbb{A}_{2,\abs{p}}$ (in particular, $\psi(p)=p$). Then for some universal constant $C>0$,
	$$n<\frac{\log\abs{p}}{2\pi} K^{1/C}.$$  
\end{lmm}
\begin{proof}
	Let $\psi_t, t\in[0,1]$ be an isotopy relative to $\{0,1\}$ such that $\psi_0=\id$ and $\psi_1=\psi$. Since the \tei\ metric on the 4-punctured sphere coincides with the hyperbolic metric, we have to bound from below the length of the geodesic segment in the homotopy class of the curve $\psi_t(p), t\in[0,1]$.
	
	This length is commensurable with the length $d$ of the geodesic segment in $\mathbb{D}_1^\infty\setminus\{\infty\}$, hence $\log K\geq Cd/2$ for some universal constant $C>0$. Lifting this segment to the right half plane via $e^{z}$ we obtain a geodesic segment between points with the real parts equal to $\log\abs{p}$ and the difference between their imaginary parts equal to $2\pi n$. Then
	$$d=2\arsinh\frac{\pi n}{\log\abs{p}}$$
	and
	$$K\geq (e^{d/2})^C>(2\sinh{d/2})^C=\left(\frac{2\pi n}{\log\abs{p}}\right)^C.$$
	
\end{proof}

\section{Asymptotic area property}
\label{sec:AAP}

In this section we discuss in more details the entire functions having asymptotic area property (AAP).

Let $f\in\mathcal{B}$, $D\supset\SV(f)$ be an open bounded set and denote $\mathcal{E}_r:=f^{-1}(\overline{\mathbb{D}}_r\setminus D)$.
Consider the function
$$I_1(\rho,D):= \frac{1}{2\pi}\iint\displaylimits_{\{\rho\leq\abs{z}\}\bigcap\mathcal{E}_\rho}\frac{dx dy}{\abs{z}^2}.$$

Recall from the Introduction (Definition~\ref{defn:as_area_property}) that $f$ has AAP relative to an open set $D\supset\SV(f)$ if
$$\limsup_{\rho\to\infty}I_1(\rho,D)<\infty,$$
and $f$ has AAP if it has AAP relative to every open set $D\supset\SV(f)$. It is easy to see from this definition that it is enough to check AAP only for bounded $D$. Further, in the case when $f$ has finite type, one can restrict only to $D$'s being the union of arbitrarily small disjoint disks around singular values.

It is convenient to have some bound for $I_1(\rho, D)$ which is independent of $D$. Proposition~\ref{prp:AAP_for_finite_type} justifies this approach at least in some generality.

\begin{defn}[Degeneration function]
	\label{defn:area_degeneration_function}
	Let $f\in\mathcal{B}$ have AAP. We will say that $\chi:\mathbb{R}_+\to\mathbb{R}_+$ is the (area) \emph{degeneration function} for $f$ if for every open $D\supset\SV(f)$,
	$$\limsup_{\rho\to\infty}\frac{I_1(\rho,D)}{\chi(\rho)}<\infty.$$
\end{defn}

It is clear from the definition that the choice of the degeneration function is not unique.

We are mainly interested in the setup when the degeneration function tends to $0$ as $\rho\to\infty$. Sometimes it is possible to provide a very precise asymptotics for $I_1$. A particular example should be more enlightening here.

\begin{example}
	\label{eg:exp_area_property}
	The exponential function $f(z)=e^z$ has AAP with the degeneration function equal to $\log\rho/\rho$.
\end{example}
\begin{proof}
	It is enough to prove the statement for $D=\mathbb{D}_R$ where $R>0$. If $\rho>R$, then $\mathcal{E}_\rho$ is a vertical strip bounded by the straight vertical lines $\Re z=\log R$ and $\Re z=\log\rho$. Therefore, it is sufficient to prove the bound for $R=1$.
	
	For $r\geq\rho$, the angular measure $\theta(r)$ of $\mathcal{E}_\rho$ in $\mathbb{S}_r$ is equal to $2\arcsin(\log\rho/r)$. Since $\log\rho$ is much smaller than $\rho$ and hence much smaller than $r$ as $\rho\to\infty$, we have $\theta(r)\sim 2\log\rho/r$. Finally,
	$$I_1(\rho,\mathbb{D}):= \frac{1}{2\pi}\iint\displaylimits_{\{\rho\leq\abs{z}\}\bigcap\mathcal{E}_\rho}\frac{dx dy}{\abs{z}^2}=\int\limits_{\rho\leq r}\frac{\theta(r)dr}{r}\sim\int\limits_{\rho\leq r}\frac{2\log\rho dr}{r^2}=\frac{2\log\rho}{\rho}.$$
\end{proof}

Similar computations show that any function $p\circ\exp$ where $p$ is a non-constant polynomial, as well as $\cos z$ and $\sin z$, have the degeneration functions equal to $\log\rho/\rho$ while for $\exp\circ p$, it is $\log\rho/\rho^{\deg p}$.

More generally, the following lemma holds.

\begin{lmm}[Degeneration for structurally finite functions]
	\label{lmm:degeneration_structurally finite}
	Let $f$ be structurally finite, i.e., having the form
	$$C+\int_0^z p(w)e^{q(w)}dw$$
	where $p$ and $q$ are polynomials, $\deg q>0$ and $C$ is a constant. Then $f$ has AAP and for every $0<\kappa<1$, $\rho^{-\kappa}$ is the degeneration function for $f$.
\end{lmm}
\begin{proof}
	Pick some $0<\kappa<1$. Let $d>1$ be the degree of $q$ and $\alpha$ be its leading coefficient. Denote for $1\leq k\leq d$,
	$$\theta_k:=\frac{(2k+1)\pi-\arg\alpha}{d}.$$
	It is well-known that $f$ has exactly $d$ (finite) asymptotic values
	$$s_k:=C+\lim\limits_{t\to\infty}\int_0^{te^{i\theta_k}} p(w)e^{q(w)}dw.$$
	Moreover, as $z\to\infty$ so that $\abs{\theta_k-\arg z}\leq\pi/d$, there is the following asymptotic presentation \cite[Lemma 4.1]{StructFinite}
	\begin{equation}
		\label{eq:struc_finite_presentation}
		f(z)=s_k+(1+o(1))\frac{p(z)}{q'(z)}e^{q(z)},
	\end{equation}
	where the bound $o(1)$ depends on $\abs{z}$.
	
	For $R>0$, consider the ``sector''
	$$S_k(R):=\left\{z\in\mathbb{C}: \abs{z}>R, \abs{\theta_k-\arg z}<\frac{\pi}{2d}-\frac{1}{\abs{z}^\kappa}\right\}.$$
	If $z\in S_k(R)$ and $R$ is big,
	$$\abs{\pi-\arg q(z)}<\frac{\pi}{2}-\frac{1}{2\abs{z}^\kappa},$$
	hence $\Re q(z)<-\abs{z}^{1-\kappa}/2$. The presentation~(\ref{eq:struc_finite_presentation}) implies then immediately that $\abs{f(z)-s_k}$ can be made arbitrarily small if $R$ is big.
	
	Analogously, denote for $1\leq k\leq d$
	$$\theta_k':=\frac{2k\pi-\arg\alpha}{d}$$
	and
	$$S_k'(R):=\left\{z\in\mathbb{C}: \abs{z}>R, \abs{\theta_k'-\arg z}<\frac{\pi}{2d}-\frac{1}{\abs{z}^\kappa}\right\}.$$
	By the similar argument as above, if $z\in S_k'(R)$ and $R$ is big enough,
	$$\abs{f(z)}>e^{\abs{z}^{1-\kappa}/3}>\abs{z}.$$
	
	Thus, the function $\chi(\rho)$ defined as the cylindrical area of $\mathbb{D}_\rho^\infty\setminus\cup_1^d(S_k\cup S_k')$ is a degeneration function for $f$. An easy computation shows that $\chi(\rho)=O(\rho^{-\kappa})$.
	
\end{proof}

Along with the integral $I_1(\rho,D)$ we will need to consider its more general version depending on a parameter $\alpha>0$:
$$I_\alpha(\rho,D):= \frac{1}{2\pi}\iint\displaylimits_{\{\alpha\rho\leq\abs{z}\}\bigcap\mathcal{E}_\rho}\frac{dx dy}{\abs{z}^2}.$$
Clearly, the value of $I_\alpha$ at $\rho$ coincides with the value of $I_1$ at $\rho$, but computed for the function $g(z):=f(\alpha z)$. It is natural to expect that $g$ and $f$ have $AAP$ with commensurable degeneration functions. This is indeed the case.

\begin{lmm}
	\label{lmm:AAP_for_scaling}
	Let $f\in\mathcal{B}$ have AAP with respect to $D$. Then for every $\alpha>0$,
	$$I_\alpha(\rho,D)<I_1(\alpha\rho,D)+2I_1(\alpha^2\rho,D)$$
	whenever $\rho$ is big enough.
\end{lmm}
\begin{proof}
	If $\alpha>1$, then trivially $I_\alpha(\rho,D)< I_1(\alpha\rho, D)$. Thus, we focus on the case $\alpha<1$.
	
	Denote $\beta:=-\log\alpha$ and fix some $\rho_0$ such that $\mathbb{D}_{\rho_0}\supset\SV(f)$. Let us switch to the logarithmic coordinates with $F$ being some logarithmic transform of $f$. Consider a parametrized nested family $\{T_x\}_{x\geq\log\rho_0}$ of tracts $T_x$ such that $F(T_x)=\mathbb{H}_x$ and let $\mathcal{T}$ be the set of all such families modulo vertical translation by $2\pi$. Recall that the pull-back of the cylindrical metric under the exponential map is Euclidean metric and denote by
	$\nu(S,a)$ the area of a measurable set $S$ inside of the right half-plane $\mathbb{H}_a$. Then we can write
	$$I_\alpha(e^x,D)=I_1(e^{x-\beta},D)+\sum\limits_{\mathcal{T}}\nu(T_{x-\beta}\setminus T_x,x-\beta)$$
	Therefore, in order to prove the lemma, it is enough to show that for big enough $x$ (independent on the choice of the family in $\mathcal{T}$), holds the inequality
	\begin{equation}
		\label{eqn:ineq_area_of_tracts_pieces}
		\nu(T_{x-\beta}\setminus T_x,x-\beta)<2\nu(T_{x-3\beta}\setminus T_{x-2\beta},x-2\beta).
	\end{equation}
	Indeed, after summing up the right hand side of (\ref{eqn:ineq_area_of_tracts_pieces}) over all families in $\mathcal{T}$, we obtain
	$$2\sum\limits_{\mathcal{T}}\nu(T_{x-3\beta}\setminus T_{x-2\beta},x-2\beta)<2I_1(e^{x-2\beta},D).$$
	
	Let us prove the inequality~(\ref{eqn:ineq_area_of_tracts_pieces}). By a small abuse of notation we assume that $F=F|_{T_{\log\rho_0}}$, i.e., $F$ is univalent, and for every $y\in\mathbb{R}$, consider three horizontal segments $s_y^1:=[x-\beta,x]\times\{y\}$, $s_y^2:=[x-3\beta,x-2\beta]\times\{y\}$ and $s_y:=[x-3\beta,x]\times\{y\}$. Let $Y$ be the set of all $y\in\mathbb{R}$ such that $F^{-1}(s_1^y)$ has a non-empty intersection with the strip $\{x-\beta<\Re z<x\}$. By making $x$ big enough, we might assume that $\abs{(F^{-1})'(w)}<1/3$ when $\Re w>x-3\beta$. Then the length of $F$ of $F^{-1}(s_y)$is smaller than $\beta$, hence for every $y\in Y$, the curve $F^{-1}(s_y^2)$ is contained in $\mathbb{H}_{x-2\beta}$. On the other hand, if $x$ is much bigger than $\log\rho_0$, due to Koebe distortion theorem applied to $F^{-1}$ and a big disk centered at $x+iy$, the derivatives $\abs{(F^{-1})'(w)}$ are uniformly commensurable along every $s_y$ (e.g., up to a multiplier $\sqrt{2}$). This provides the desired bound on the area and finishes the proof of the lemma.  
\end{proof}

More easily, if $f$ has AAP for $D$, then for every $b\in\mathbb{C}$, $f_b(z):=f(z-b)$ has AAP for $D$, and for big enough $\rho$, the values $I_1(\rho,D)$ computed for $f_b$ do not exceed $MI_\alpha(\rho,D)$ computed for $f$ where $M>1,\alpha<1$ are some constants. 

If $f$ is of finite type, using similar techniques we can prove an even stronger result.

\begin{prp}[AAP for finite type functions]
	\label{prp:AAP_for_finite_type}
	Let $f$ be a finite type entire function with bounded degrees of critical points and $D=\cup_{v\in\SV(f)}D_v$ where $D_v\ni v$ are bounded Riemann domains with pairwise disjoint closures.
	
	If $f$ has AAP relative to $D$, then it also has AAP relative to any other open set $D'\subset D$ containing all singular values.
	
	Moreover, for big enough $\rho$ and some constants $M>1$, $\alpha<1$,
	$$I_1(\rho, D')<MI_\alpha(\rho,D).$$ 
\end{prp}
\begin{proof}[Sketch of proof]
	Choose some pairwise disjoint bounded Riemann domains $\hat{D}_v$ such that $\overline{D}_v\subset\hat{D}_v$ and let $R_v:\hat{D}_v\to\mathbb{D}$ be a Riemann map of $\hat{D}_v$ mapping $v$ to $0$. Denote $D'_v:=D_v\cap D'$ and fix some numbers $\beta<\alpha<1$ such that for every singular value $v$, $R_v(D'_v)\subset\mathbb{D}_\beta$ and $R_v(D_v)\subset\mathbb{D}_\alpha$. Without loss of generality we might assume that $D_v=R_v^{-1}(\mathbb{D}_\alpha)$ and $D'_v=R_v^{-1}(\mathbb{D}_\beta)$.
	
	In the setting as above we can switch to the ``semi-logarithmic coordinates'' in a sense that for every $v\in\SV(f)$, there is a map $F_v:=R_v\circ f\circ\exp$ defined on a disjoint union $\mathcal{T}_v$ of Riemann domains. The setup is well-defined because we are only interested in the pre-images (or their parts) of $\mathbb{D}$ under $R_v\circ f$ which are far from the origin. We need to consider $3$ cases depending on the type of the branched covering $F_v|_U$, where $U$ is a connected component of $\mathcal{T}_v$. We are going to work with each $v$ separately so let us suppress $v$ from the indices and locally use the notation $F=F_v|_U$.
	
	\begin{enumerate}
		\item[(Regular value)] Let $F$ be a conformal map. Then from the Koebe 1/4-theorem and $2\pi i$-periodicity of the exponential map follows that $\abs{(F^{-1})'(0)}$ is universally bounded. Therefore, due to Koebe distortion theorem, the diameters of $F^{-1}(\mathbb{D}_\alpha)$ are uniformly bounded and the area of $F^{-1}(\mathbb{A}_{\beta,\alpha})$ is commensurable with the are of $F^{-1}(\mathbb{A}_{\alpha,\sqrt{\alpha}})$. This means that the space in $D'$ (compared to $D$) added inside of the regular pre-image domains has area commensurable to the area already included into $I_\alpha(\rho,D)$ for some $\alpha<1$ and corresponding to a regular pre-image.
		\item[(Critical value)] Let $F$ be a branched covering of degree $d$ with the only critical point $p:=F^{-1}(0)$. Then there exists a lift of $F$ of the form $F^{1/d}$ and its inverse is a conformal map of $\mathbb{D}$. Since the degrees of critical points are bounded, we can proceed as in the previous case.
		\item[(Asymptotic value)] Let $F:U\to\mathbb{D}\setminus\{0\}$ be a covering of infinite degree. We can switch to the genuine logarithmic coordinates by considering a conformal map $\tilde{F}:U\to\mathbb{H}_0$ defined by the relation $\tilde{F}=-\log F$. Then, similarly as in the proof of Lemma~\ref{lmm:AAP_for_scaling}, by Koebe distortion argument, the $\tilde{F}$-pre-images of segments $[-\log\alpha/2,-\log\beta]\times\{y\}$ have lengths bounded independently of $y\in\mathbb{R}$. For the same reason, the area distortion near such segments is bounded, and the claim follows.
	\end{enumerate}	
\end{proof}

Note that $AAP$ generically behaves well under composition of functions. That is, if $f$ and $g$ have AAP, then it is natural to expect that $f\circ g$ also has AAP. Indeed, if we switch to the logarithmic coordinates $F,G$, then the pull-back of the cylindrical measure is the Euclidean measure. So, if the tracts of $F$ ``fill'' almost all space near $+\infty$, their preimage under $G$ should ``fill'' most of the area in the tracts of $F$. For example, a very rough estimate shows that $e^{e^z}$ has AAP with the degeneration function $\log^2\rho/\rho$.

\section{Koebe-like estimates for \qc\ maps with small dilatation per area}
\label{sec:Koebe}

We will prove two quantitative estimates for the \qc\ maps possibly having a very big maximal dilatation but supported on a small area. The computations rely heavily on the techniques from the proof of the \tei--Wittich theorem~\ref{thm:teich--wittich} as presented in \cite[Chapter V.6]{LehtoVirtanen}.

\begin{lmm}[Conditional Koebe distortion]
	\label{lmm:conditional_Koebe_I}
	Let $\varphi:\mathbb{D}\to U\subset\mathbb{C}$ be a \qc\ map such that $\varphi(0)=0$ and
	$$I= \frac{1}{2\pi}\iint\displaylimits_{0<\abs{z}< 1}\frac{D_\varphi(z)-1}{\abs{z}^2}dx dy< \infty.$$
	
	If we restrict to $\varphi$ such that $I\leq\kappa$ for some parameter $\kappa>0$, then:
	\begin{enumerate}
		\item for every $z\in \mathbb{D}$, $\abs{\varphi(z)}/\abs{z\varphi'(0)}$ is bounded from below by a constant depending only on $\kappa$;
		\item there exists a radius $0<r_\kappa<1$ such that for every $z\in\mathbb{D}_{r_\kappa}(0)$, $\abs{\varphi(z)}/\abs{z\varphi'(0)}$ is bounded from above by a constant depending only on $\kappa$.
	\end{enumerate}
	
	Moreover, as $\kappa\to 0$, the radii $r_\kappa$ can be chosen in such a way that for every $z\in\mathbb{D}_{r_\kappa}(0)$,
	$$\left|\frac{\abs{\varphi(z)}}{\abs{z\varphi'(0)}}-1\right|<C_\kappa$$
	where $C_\kappa\to 0$ as $\kappa\to 0$.	
\end{lmm}
\begin{proof}
	From the Teichm\"uller--Wittich Theorem~\ref{thm:teich--wittich}, we know that $\varphi$ is conformal at $0$, hence after rescaling we may assume that $\varphi'(0)=1$, i.e., $\abs{\varphi(z)}/\abs{z}\to 1$ as $\abs{z}\to 0$.
	
	Let $0<\delta<\rho<1$. From Lemma~\ref{lmm:ineq_mod_diff}, we obtain 
	\begin{equation}
		\label{eqn:lmm_conditional_Koebe}
		\abs{\mod\varphi(\mathbb{A}_{\delta,\rho})-\log\frac{\rho}{\delta}}\leq\frac{1}{2\pi}\iint\displaylimits_{\mathbb{A}_{\delta,\rho}} \frac{D(z)-1}{\abs{z}^2}dxdy\leq \kappa.
	\end{equation}
	From Theorem~\ref{thm:essential_round_annulus}, if $\delta$ is small enough, there exists a round annulus $B$ centered at $0$ so that $\mod\varphi(\mathbb{A}_{\delta,\rho})=\mod B+O(1)$. We may assume that $B$ is the maximal such annulus, i.e., its outer radius is equal to $R=\min_{z\in \partial\mathbb{D}_\rho}\abs{\varphi(z)}$ and its inner radius to $r=\max_{z\in \partial\mathbb{D}_\delta}\abs{\varphi(z)}$. Due to the conformality at $0$, by making $\delta$ small we have $(1-\varepsilon)\delta<r<(1+\varepsilon)\delta$ for any initially chosen $\varepsilon>0$.
	
	Thus, as $\varepsilon\to 0$ the inequality~(\ref{eqn:lmm_conditional_Koebe}) rewrites as
	$$\log\rho-\kappa+O(1)\leq\log R\leq\log\rho+\kappa+O(1),$$
	which implies that $\min_{z\in \partial\mathbb{D}_\rho}\abs{\varphi(z)}/\rho$ is bounded from below by a constant depending only on $\kappa$ (but note that it is also bounded from above by $e^{\kappa+O(1)}$). This proves the first part of the statement.
	
	Choose (if possible) some $\rho_1>\rho$ such that $\mod \mathbb{A}_{\rho,\rho_1}=I+C$ where $C$ is the universal constant from Theorem~\ref{thm:essential_round_annulus}. Then from the inequality 
	$$\abs{\mod\varphi(\mathbb{A}_{\rho,\rho_1})-\mod \mathbb{A}_{\rho,\rho_1}}\leq\frac{1}{2\pi}\iint\displaylimits_{\mathbb{A}_{\rho,\rho_1}} \frac{D(z)-1}{\abs{z}^2}dxdy\leq I,$$
	we see that $\mod\varphi(\mathbb{A}_{\rho,\rho_1})>C$, hence there exists a circle centered at $0$ and contained in $\overline{\varphi(\mathbb{A}_{\rho,\rho_1})}$. This means that $$\max_{z\in \partial\mathbb{D}_\rho}\abs{\varphi(z)}\leq\min_{z\in \partial\mathbb{D}_{\rho_1}}\abs{\varphi(z)}=\rho_1 e^{I+O(1)}=\rho e^{C+I}e^{I+O(1)}.$$
	
	The estimate for $\kappa\to 0$ follows almost immediately from the proof of \cite[Hilfssatz V.6.1]{LehtoVirtanen} if we upgrade the input data. More precisely, we no longer need the maximal dilatation $K$ to estimate the quantity $\psi(r):=\max_{\abs{z}=r}\abs{\varphi(z)}/\min_{\abs{z}=r}\abs{\varphi(z)}$ (notation of \cite{LehtoVirtanen}). Instead, use the uniform bounds obtained from the first part of the lemma, i.e., $\psi(r)<C_1$ for $r>r_\kappa$ and some constant $C_1$ depending only on $\kappa$. Then it follows from \cite[Chapter V, inequality (6.21)]{LehtoVirtanen} together with the discussion in the subsequent paragraph, that $\psi(r)$ is smaller than $\epsilon=\epsilon(\kappa)$ (tending to $0$ as $\kappa\to 0$) if $r<r_\kappa$ where $r_\kappa$ depends only on $\kappa$. Then the required statement can be proved exactly as the first part of the lemma: the universal constant from Theorem~\ref{thm:essential_round_annulus} can be replaced by a constant arbitrarily close to $0$.
\end{proof}

Before proving a similar statement for the angular distortion, we need a short preparatory lemma. For (a small) $d>0$, denote by $R_d$ the rectangle $[0,d]\times[0,1]$ and consider the situation when such rectagle is divided into two quadrilaterals by an injective path $\gamma:[0,1]\to R$ contained in the interior of $R_d$ except of its endpoints belonging to different vertical sides of $R_d$. The upper and the lower quadrilaterals are denoted by $Q_1$ and $Q_2$, respectively. We assume that the orientation is chosen in such a way the $\gamma$ and the horizontal sides of $R_d$ are the corresponding $a$-sides of $Q_1,Q_2$ and $R_d$.

\begin{lmm}
	\label{lmm:splitted_rectangle}
	Fix some $0<\tau<1$. For every $\varepsilon>0$, there exist $d_0>0,\delta>0$ such that if simultaneously $\mod Q_1<d(1+\delta)/(1-\tau)$, $\mod Q_2<d(1+\delta)/\tau$ and $d<d_0$, then the path $\gamma$ is contained inside of a horizontal strip of height at most $\varepsilon$. 
\end{lmm}
\begin{proof}
	Let us consider the annulus $A$ obtained by gluing $R_d$ together with its mirror copy along the vertical sides. Then the union $\Gamma$ of $\gamma$ with its mirror copy is a topological circle dividing $A$ into two annuli $A_1$ and $A_2$, each of them being the quadrilaterals $Q_1$ and $Q_2$, respectively, glued with their mirror copies.
	
	Then $\mod A=\pi/\mod R$ and $\mod A_i=\pi/\mod Q_i$, $i=1,2$ (the relations follow immediately from \cite[Hilfssatz 6.5]{LehtoVirtanen} after noticing that $A_1$ and $A_2$ have an axis of symmetry). Due to Theorem~\ref{thm:essential_round_annulus}, if $\mod A_i$ is big enough (that is, when $d$ is small enough), $A_i$ contains a round annulus $B_i$ such that $\mod A_i-\mod B_i<C$ for some universal constant $C>0$. Then the curve $\Gamma$ is contained inside of the round annulus $B'$ between $B_1$ and $B_2$. However, by superadditivity of modulus,
	$$\mod B'\leq\mod A-\mod B_1-\mod B_2<$$
	$$2C+\mod A-\mod A_1-\mod A_2=$$
	$$2C+\pi\left(\frac{1}{\mod R_d}-\frac{1}{\mod Q_1}-\frac{1}{\mod Q_2}\right)<$$
	$$2C+\frac{\pi}{d}\left(1-\frac{1}{1+\delta}\right)<2C+\frac{\pi\delta}{d}.$$
	Therefore
	$$\frac{\mod B'}{\mod A}<\frac{2Cd+\pi\delta}{\pi}\to 0$$
	as $\delta\to 0$ and $d\to 0$. Since $B'$ and $A$ are concentric round annuli, the claim follows.
\end{proof}

Now, we state and prove a key result that allows us to maintain and reproduce the ``invariant structure'' in Theorem~\ref{thm:invariant_structure}. It says, that if the cylindrical integral is small for a \qc\ map fixing $0$, then this map is predictably close to identity on a neighbourhood of $0$ which is independent of the maximal dilatation, and depends only on the value of the integral.

\begin{prp}[Distortion of identity]
	\label{prp:distortion of identity}
	For every $\varepsilon>0$, there exist $0<\kappa<\infty$ and a radius $0<r<1$, so that the following statement holds.
	
	If $\varphi:\mathbb{D}\to U\subset\mathbb{C}$ is a \qc\ map such that $\varphi(0)=0$, $\varphi'(0)=1$ and
	$$\frac{1}{2\pi}\iint\displaylimits_{0<\abs{z}< 1}\frac{D_\varphi(z)-1}{\abs{z}^2}dx dy<\kappa,$$
	then for $z\in\mathbb{D}_r\setminus\{0\}$,
	$$d_{\cyl}(\varphi(z),z)<\varepsilon.$$
\end{prp}
\begin{proof}
	The first part of the proposition, about the radial distortion, is already proven in Lemma~\ref{lmm:conditional_Koebe_I}. To provide bounds for the angular distortion, let us switch to the logarithmic coordinates. This can be described by the following diagram where by $\mathbb{H}$ we denote the left half-plane $\{z:\Re z<0\}$. The map $\xi$ is defined up to a vertical translation by $2\pi$, so we fix some choice of it. 
	\begin{center}
		\begin{tikzcd}
			\mathbb{H} \arrow[r, "{\xi}"] \arrow[d, "\exp"]	& \mathbb{C} \arrow[d, "\exp"] \\
			\mathbb{D}\setminus\{0\} \arrow[r, "{\varphi}"] & \mathbb{C}\setminus\{0\}
		\end{tikzcd}
	\end{center}
	\vspace{0.5cm}
	
	If we denote
	$$I(\rho):= \frac{1}{2\pi}\iint\displaylimits_{0<\abs{z}< \rho}\frac{D(z)-1}{\abs{z}^2}dx dy,$$
	in the logarithmic coordinates the integral inequality rewrites as
	$$I^{\log}(\chi):=I(e^\chi)=\frac{1}{2\pi}\iint\displaylimits_{S_{\chi}}\left(D(e^\chi)-1\right)dx dy\leq I^{\log}(0)\leq\kappa,$$
	where $\rho=e^\chi$ , $\chi\in[-\infty,0]$ and $S_\chi=[-\infty,\chi]\times[0,2\pi]$. Clearly, $S_{\chi}$ can be translated vertically without changing $I^{\log}(\chi)$. Let us agree that for a curve, the difference between the supremum and the infimum of the imaginary parts of points belonging to the curve is called \emph{height} of the curve.
	
	First, given $\epsilon>0$, we show existence of $\kappa$ and $\chi$ such that if $I^{\log}(0)<\kappa$ and $x<\chi$, then $\abs{\Im\xi(x+iy)-y}<\varepsilon$ for at least one $y=y(x)\in[0,2\pi]$. Consider a (very long) rectangle $[x_1,x]\times[0,2\pi]$ such than $(x-x_1)/2\pi\in\mathbb{N}$ and recall from Lemma~\ref{lmm:conditional_Koebe_I} that we have the upper and the lower bound on $\Re\xi(x)$ if $x$ is close enough to $-\infty$. At the same time, we can choose $x_1$ even closer to $-\infty$, in the region where we have precise estimates for the map $\xi$ due to conformality of $\varphi$ at $0$, that is, by increasing $\abs{x_1}$ we can make $\abs{\Im\xi(x_1+iy)-y}$ arbitrarily small (but here $x_1$ depends also on a particular map $\varphi$). If for some $y\in[0,2\pi]$, we have $\Im\xi(x+iy)=y$, the claim is proven. Otherwise, $\Im\xi(x+iy)-y$ has the same sign for all $y\in[0,2\pi]$. Then we can literally repeat the computations in \cite[p.241-242]{LehtoVirtanen} for the skewed (by translating the side $\{x\}\times[0,2\pi]$ vertically by $x-x_1$) quadrilateral and obtain the upper bound on $\min_{y\in[0,2\pi]}\abs{\Im\xi(x+iy)-y}$ which depends only on $\kappa$ and tends to $0$ as $\kappa\to 0$ and $\chi\to-\infty$. Thus, more generally, we have shown that if $x<\chi$, there exists $y\in[0,2\pi]$ (depending on $x$) so that $\abs{\xi(x+iy)-x-iy}<2\epsilon$. We will use this statement at the end of the proof.
	
	Now, fix a (small) number $d>0$ and consider a rectangle $R:=[x,x+d]\times[0,2\pi]$. Then, if $x$ is smaller than some $\chi$, due to Lemma~\ref{lmm:conditional_Koebe_I} and $2\pi i$-periodicity of $\xi$, the area of $\xi(R)$ does not exceed $2\pi(d+\epsilon_1)$, where $\epsilon_1$ depends only on $\kappa$ and tends to $0$ as $\kappa\to 0$ and $\chi\to-\infty$. Let us subdivide $R$ into $n>0$ equal ``horizontal'' rectangles with the lengths of the horizontal sides equal to $d$ and the lengths of the vertical sides equal to $2\pi/n$. Let $R_1$ be one of them, such that its $\xi$-image has area not exceeding $2\pi(d+\epsilon_1)/n$.
	
	Let $s_b=s_b(\xi(R))$ be the distance between the b-sides of the quadrilateral $\xi(R_1)$. Applying the left side of Rengel's inequality~(\ref{lmm:Rengel}) to $\xi(R_1)$, we obtain
	$$\mod \xi(R_1)\geq\frac{s_b^2\left(\xi(R_1)\right)}{m\left(\xi(R_1)\right)}\geq\frac{ns_b^2}{2\pi(d+\epsilon_1)}.$$ 
	
	On the other hand, from Lemma~\ref{lmm:ineq_mod_rectangle}, we have
	$$\mod \xi(R_1)\leq\frac{n^2}{(2\pi)^2}\iint\displaylimits_{R'} D_\xi(z) dxdy\leq$$
	$$\frac{n^2}{2\pi}\left(I^{\log}(x+d)-I^{\log}(x)\right)+\frac{nd}{2\pi}<\frac{n}{2\pi}\left(d+n\kappa\right).$$
	
	Combining the two inequalities above, we obtain an estimate on $s_b$:
	$$s_b< \sqrt{(d+\epsilon_1)(d+n\kappa)}\leq d+\max\{n\kappa,\epsilon_1\}.$$
	Note that the distance between images under $\xi$ of the lines $\Re z=x$ and $\Re z=x+d$ is bigger than $d-\epsilon_2$ for some $\epsilon_2>0$. That is, there exists a curve $\gamma_a$ joining the $b$-sides of $\xi(R_1)$ such that its height is smaller that some $\epsilon_3$ which can be made arbitrarily small by adjusting $\kappa$ and $\chi$. On the other hand, $\xi^{-1}(\gamma_a)$ is contained in $R_1$ and hence has height not exceeding $2\pi/n$. Note that $\gamma_a$ can be parametrized by the interval $[0,d]$ in such a way that $\abs{\gamma_a(t)-\gamma_a(0)-t}<\epsilon_4$, where $\epsilon_4$ can be arbitrarily close to $0$. Instead of $R$, we now consider the perturbed quadrilateral $R'$ with the $a$-sides being the curve $\xi^{-1}(\gamma_a)$ and its $2\pi i$-translation and the $b$-sides being the corresponding straight vertical segments.
	
	Now, let $s_a=s_a(\xi(R'))$ be the distance between the a-sides of the quadrilateral $\xi(R')$. From their definition, $s_a\geq 2\pi-2\epsilon_4$. Applying the right side of Rengel's inequality~(\ref{lmm:Rengel}) to $\xi(R)$, we obtain
	$$\mod \xi(R_2)=\frac{1}{\mod\xi(R')}\geq\frac{s_a^2\left(\xi(R')\right)}{m\left(\xi(R')\right)}\geq\frac{s_a^2}{2\pi(d+\epsilon_1)}$$
	where $R_2$ is the quadrilateral $R'$ with the reversed orientation of sides and $\epsilon_1$ can be chosen small compared to $d$. 
	
	From Lemma~\ref{lmm:ineq_mod_rectangle} applied to $\xi(R_2)$, we obtain
	$$\mod \xi(R_2)\leq\frac{1}{d^2}\iint\displaylimits_{R'} D_\xi(z) dxdy\leq$$
	$$\frac{2\pi}{d^2}\left(I^{\log}(x+d)-I^{\log}(x)\right)+\frac{2\pi}{d}<\frac{2\pi}{d}\left(1+\frac{\kappa}{d}\right).$$
	
	Combining the two inequalities, we obtain:
	$$s_a<2\pi \sqrt{(1+\epsilon_1/d)(1+\kappa/d)}\leq 2\pi+2\pi\max\{\kappa,\epsilon_1\}/d.$$
	Exactly as for $s_b$, this means that inside of the quadrilateral $\xi(R')$ there is curve $\gamma_b$, parametrized by the interval $[0,2\pi]$ such that $\abs{\gamma_b(t)-\gamma_b(0)-t}<\epsilon_5$, where $\epsilon_5$ can be arbitrarily close to $0$. Furthermore, we can assume that $\Re\xi^{-1}(\gamma_b(0))=\Re\xi^{-1}(\gamma_b(2\pi))$, otherwise augment $\xi^{-1}(\gamma_b)$ by a small part of $\xi^{-1}(\gamma_a)$ (and truncate $\xi^{-1}(\gamma_a)$, respectively).
	
	Let us reformulate what we have shown up to this moment. Given $d>0$ and $\epsilon>0$, there are such $\kappa$ and $\chi$ that the following statement holds. If $x<\chi$, there exists $y\in[0,2\pi]$ such that the rectangle $R$ of width $d$ and height $2\pi$ with the lower left vertex $x+iy$ can be $\epsilon$-approximated by a quadrilateral $Q$ (i.e., each side of $Q$ is in the $\epsilon$-neighbourhood of the corresponding side of $R$) such that $\xi(Q)$ is an $\epsilon$-approximation of a translated by $\tau\in\mathbb{C}$ copy $R_\tau$ of $R$. The quadrilaterals $Q$ corresponding to $R+2\pi im, m\in\mathbb{Z}$ an $R-kd, k\in\mathbb{N}$ have mutually disjoint interior and cover without gaps some left half-plane. Moreover, the sides of $\xi(Q)$ can be parametrized by the sides of $R_\tau$ via a function $\Pi:\partial R_\tau\to\partial\xi(Q)$, respecting the sides, in such a way that $\abs{\Pi(z)-z}<\epsilon$.
	
	Now, we want to improve the statement above by replacing the height $2\pi$ by an arbitrarily small height $h$ (subject to good enough $\kappa$ and $\chi$). Pick some $0<\tau<1$ and cut $Q$ by a segment $L$ of the horizontal straight line $\Im z=y+2\pi\tau$ into two quadrilaterals $Q_{1-\tau}$ and $Q_\tau$, containing the ``upper'' and the ``lower'' side of $Q$, respectively. Since $Q$ is an $\epsilon$-approximation of $R$, the modulus of $Q_\tau$ is close to $d/(2\pi\tau)$ (see, e.g., \cite[p. 29, Satz \"uber die Modulkonvergenz]{LehtoVirtanen}).
	
	Applying Lemma~\ref{lmm:ineq_mod_rectangle} to $Q_\tau$, we obtain
	$$\mod\xi(Q_\tau)\leq\frac{1}{(2\pi\tau-\epsilon)^2}\iint\displaylimits_{Q_\tau} D_\xi(z) dxdy<$$
	$$\frac{1}{(2\pi\tau-\epsilon)^2}\left(2\pi\kappa+(2\pi\tau+\epsilon)(d+2\epsilon)\right)<(1+\epsilon_6)\mod Q_\tau,$$
	where $\epsilon_6\to 0$ as $\epsilon\to 0$. The same estimate holds for $Q_{1-\tau}$.

	Let $\zeta:\xi(Q)\to \tilde{R}$ be the canonical conformal map from $\xi(Q)$ to the rectangle $\tilde{R}$ having height $2\pi$. Then $\mod \zeta\circ\xi(Q_t)$ and $\mod \zeta\circ\xi(Q_{1-\tau})$ are smaller than  $d/(2\pi\tau)+\epsilon_7$ and $d/(2\pi(1-\tau))+\epsilon_7$, respectively, and $\mod\zeta\circ\xi(Q)$ is close to $d/(2\pi)$. From Lemma~\ref{lmm:splitted_rectangle} follows that $\zeta\circ\xi(L)$ is contained in a horizontal strip of height at most $\epsilon_8>0$ which can be made arbitrarily small by adjustments of $\kappa$ and $\chi$. On the other hand, $\zeta\to\id$ uniformly as $\epsilon\to 0$ (see, e.g., \cite[Theorem 2.11]{Pommerenke_book}). Therefore, $\xi(L)$ is contained in a strip of a small height tending to $0$ as $\epsilon\to 0$ and containing the straight line $\Im z=y+2\pi\tau$.
	
	If we fix some $m>0$, this procedure can be performed for $\tau=1/m,2/m,...,(m-1)/m$ if $\kappa$ and $\chi$ are good enough. That is, up to a small error, $\xi$ translates each $R$ vertically together with the subdivision into $m$ smaller rectangles. Recall that we have shown, that on every straight vertical line there is a point $p$ (hence also $2\pi i$ translates of $p$) such that $\abs{\xi(p)-p}$ is small. It belongs to at least one of the smaller rectangles in the subdivision. Hence $\abs{\Im\xi(z)-\Im z}<2\pi/m+\epsilon_9$. By adjusting $\kappa$ and $\chi$, we can make $m$ arbitrarily big and $\epsilon_9$ arbitrarily small. This finishes the proof of the proposition. 
\end{proof}

\section{Shifts and fat spiders}
\label{sec:shifts_and_spiders}
In this section we introduce all left-over tools needed to prove the main result. In Subsection~\ref{subsec:shifts_properties} we define shifts and discuss the properties. $(K,\delta)$-regularity of tracts and fat spiders are defined in Subsection~\ref{subsec:tracts_regularity} and Subsection~\ref{subsec:spiders}, respectively.

\subsection{Shifts and their properties}
\label{subsec:shifts_properties}

To simplify the notation, we make use of the following language.

\begin{defn}[Shift]
	\label{defn:shift}
	Let $U\subset\mathbb{C}$ have a non-empty interior $U^\circ$, let a point $x$ either belong to $ U^{\circ}$ or be a puncture (hence belonging to $\partial U$), let $[\gamma]$ be a (fixed endpoints) homotopy class of paths $\gamma:[0,1]\to U^\circ\cup\{x\}$ such that $\gamma(0)=x$. We say that a homeomorphism $\psi:\mathbb{C}\to\mathbb{C}$ is a \emph{shift} (from $x$ to $\gamma(1)$) along $[\gamma]$ in $U$ if there exists an isotopy $\psi_t:\mathbb{C}\to\mathbb{C}, t\in[0,1]$ such that $\psi_t=\id$ on $\mathbb{C}\setminus(U^\circ\cup\{x\})$, $\psi_0=\id$, $\psi_1=\psi$ and $[\psi_t(x)]\in[\gamma]$.
	
	Additionally, introduce the following notation.
	\begin{enumerate}
		\item By $K_U(x,[\gamma])$ denote the extremal maximal dilatation in the \tei\ isotopy class of the shift along $[\gamma]$. We say that $\psi$ is a \emph{$K$-shift} along $[\gamma]$ if $K_U(x,[\gamma])\leq K$.
		\item For $y\in U^{\circ}\cup\{x\}$, $$K_U(x,y):=\inf_{\{\gamma:\gamma(1)=y\}} K_U(x,[\gamma])$$
		(hence for a fixed $[\gamma]$, $K_U(x,\gamma(1))\leq K_U(x,[\gamma])$).
		\item For sets $Y_1\subset U^\circ$, $Y_2\subset\mathbb{C}$,
		$$K_U(Y_1\gg Y_2):=\sup_{x\in Y_1}\inf_{y\in Y_2\cap U^{\circ}} K_U(x,y).$$
		If for some $x\in Y_1$ its path-connected component of $U^{\circ}$ does not contain a point of $Y_2$, we define
		$$K_U(Y_1\gg Y_2):=\infty.$$ 
	\end{enumerate}
\end{defn}

Note that unlike in the definition of $K_U(x,[\gamma])$ there is no initially chosen homotopy class along which $x$ is moved in $(2)-(3)$. In other words, $K_U(x,[\gamma])$ is based on the \tei\ distance in the \tei\ space of $U\setminus \{x\}$ while in $(2)-(3)$ the \tei\ distance on the moduli space of $U\setminus \{x\}$ is involved. We use the symbol ``$\gg$'' to emphasize that $Y_1,Y_2$ play asymmetric role and to avoid confusion between them.

We prove a few statements about properties of shifts which we are going to use actively in the rest of the article.

\begin{prp}
	\label{prp:K_shifts_imply_distance_bounds}
	Fix a real number $0<q<1$. Let $X\subset\mathbb{D}_q$ be a set containing at least $3$ points, $x_1,x_2\in X, x_1\neq x_2$ are isolated in $X$ and $\psi_i, i=1,2$ be a $K$-shift of $x_i$ in $\mathbb{C}\setminus X$ such that $\abs{\psi_i(x_i)}>1$. Then $\psi_1$ is isotopic relative to $X$ to a $K'$-shift $\psi'_1$ in $\mathbb{D}_{2\abs{\psi_1(x_1)}}\setminus X$ with $K'=O(K^{\beta})$ for a universal constant $\beta>0$.
	
	If, additionally, for $p>1$ and a set $Y\subset\mathbb{D}_p^\infty\setminus\{\infty\}$, $\psi_1$ is a $K$-shift along $[\gamma]$ in $\mathbb{C}\setminus(X\cup Y)$, then for the ``semi-projected'' path $\gamma_\pi:[0,1]\to\mathbb{D}\setminus X$ defined by the formula $\gamma_\pi(t):=\min\{\abs{\gamma(t)},1\}e^{i\arg\gamma(t)}$, we have
	$$K_{\mathbb{C}\setminus(X\cup Y)}(\gamma(0),[\gamma_\pi])=O(K^{\beta+4}).$$

	The bounds $O(.)$ depend only on $q$ and $p$.
\end{prp}
The maps $\psi_1$ and $\psi_2$ play symmetric role in the proposition hence all statements are valid also for $\psi_2$.
\begin{proof}
	Without loss of generality we might assume that $0\in X$ and $x_i\neq 0, i=1,2$. Otherwise one can apply a \qc\ change of coordinates equal to identity on $\mathbb{C}\setminus\mathbb{D}$ in such a way that this will increase $K$ and $K'$ at most by a multiplicative constant (depending only on $q$).
	
	Let us first show that $\abs{\psi_i(x_i)}=e^{O(K^2)}$. It is enough to do this for $X=\{0,x_1,x_2\}$. Assume that $\abs{x_1}\leq \abs{x_2}$. By Lemma~\ref{lmm:max_over_min} applied to the circle $\partial\mathbb{D}_{\abs{x_2}}$ and the map $\psi_1$ we have
	$$\abs{\psi_1(x_1)}<\max_{z\in\partial\mathbb{D}_{\abs{x_2}}}\abs{\psi_1(z)}\leq\abs{x_2}e^{\pi K}<e^{\pi K}.$$
	
	To obtain the upper bound for $\abs{\psi_2(x_2)}$, consider first the round annulus $\mathbb{A}_{\abs{x_1},\abs{x_2}}$. Its image under $\psi_1$ is an annulus separating $0$ and $\psi_1(x_1)$ from $x_2$ and $\infty$. Since $\abs{\psi_1(x_1)}>1$, its modulus is bounded from above by a constant depending only on $q$. This implies that  $\abs{x_2}=\abs{x_1}e^{O(K)}$.
	
	Next, consider the annulus $\mathbb{A}_{1,\abs{\psi_2(x_2)}}$, having modulus $\log\abs{\psi_2(x_2)}$. Its image under the map $\psi_2^{-1}$ separates $0$ and $x_1$ from $x_2$ and $\infty$. Therefore, by \tei's theorem about separating annulus \cite[Section II.1.3]{LehtoVirtanen}, we have
	$$\frac{\log\abs{\psi_2(x_2)}}{K}\leq\mod\psi_2(\mathbb{A}_{1,\abs{\psi_2(x_2)}})\leq 2\mu\left(\sqrt{\frac{1}{1+\frac{\abs{x_2}}{\abs{x_1}}}}\right)=2\mu\left(\sqrt{\frac{1}{1+e^{O(K)}}}\right),$$
	where $\mu:(0,1)\to(0,+\infty)$ is a decreasing function. From \cite[Chapter II, Equation (2.14)]{LehtoVirtanen},
	$$\mu\left(\sqrt{\frac{1}{1+e^{O(K)}}}\right)=O(\log4\sqrt{1+e^{O(K)}})=O(K),$$
	hence $\abs{\psi_2(x_2)}=e^{O(K^2)}$.
	
	Note also that we have shown along the way that
	$$\abs{x_1}=\abs{x_2}e^{-O(K)}\geq\abs{\psi_1(x_1)}e^{-O(K)}\geq e^{-O(K)}.$$
	We will need this bound a bit later. 
	
	We are finally ready prove the estimates for $\psi'_1$. To simplify the notation, let $x:=x_1$, $\psi:=\psi_1$, $\psi':=\psi'_1$ (we will not use the inequality $\abs{x_1}\leq\abs{x_2}$) and prove the proposition for $\psi$. Applying Lemma~\ref{lmm:id_near_puncture} to the disk $\mathbb{D}_{\abs{\psi(x)}}^\infty$ around $\infty$ and the map $\psi^{-1}$, we obtain a map $\varphi$ isotopic to $\psi$ relative to $X$ and equal to identity on $\mathbb{D}_r^\infty$ where $r=2\max\{\abs{\psi(x)},e^{\pi K}\}$ (as earlier, the bound $e^{\pi K}$ is obtained from Lemma~\ref{lmm:max_over_min} applied to the circle $\partial\mathbb{D}_{\abs{\psi(x)}}$). The maximal dilatation of $\varphi$ is equal to $$O(K^4)\left(1+\left|\log e^{O(K^2)}\right|\right)=O(K^6).$$ 
	If $\abs{\psi(x)}\geq e^{\pi K}$, then $\varphi(z)=\id$ on $\mathbb{D}_{2\abs{\psi(x)}}^\infty$. Otherwise, after applying Lemma~\ref{lmm:conformal_neighbourhood_expand} to the disk $\mathbb{D}_{\abs{\psi(x)}}^\infty$, we might assume that $\varphi=\id$ on $\mathbb{D}_{2\abs{\psi(x)}}^\infty$ and its maximal dilatation is equal to
	$$O(K^6)(\log e^{O(K)})^2=O(K^8).$$
	
	The map $\varphi$ cannot yet be taken as $\psi'$ because their isotopy classes might differ by an additional $n$-twist of the annulus $\mathbb{A}_{\abs{\psi(x)},2\abs{\psi(x)}}$. The multiplicity $n$ of this twist can be estimated by applying Lemma~\ref{lmm:twist_in_3_punctured_sphere} to the map $\varphi\circ\psi^{-1}$ having maximal dilatation $O(K^9)$. For this, rescale the coordinates to make $x_2=1$. Then the bound on $\abs{\psi(x)}$ stays the same:
	$$\abs{\psi(x)}=\abs{\psi_1(x_1)}/\abs{x_2}=e^{O(K^2)}e^{O(K)}=e^{O(K^2)}.$$
	Thus,
	$$n<\frac{\log\abs{\psi(x)}}{2\pi} K^{9/C}=O(K^{2+9/C}).$$
	
	After post-composing $\varphi$ with a \qc\ $n$-twist on the annulus $\mathbb{A}_{\abs{\psi(x)},2\abs{\psi(x)}}$, we obtain the desired map $\psi'$ having dilatation $K'=O(K^8)(1+O(\theta^2))=O(K^\beta)$.
	
	If $\psi$ is a $K$-shift along a path $\gamma$ in $\mathbb{C}\setminus(X\cup Y)$, then the (properly adjusted) conjugacy from the proof of Lemma~\ref{lmm:conformal_neighbourhood_expand} (applied to $\mathbb{D}_p^\infty$ and the map $\psi'$) will replace $[\gamma]$ by $[\gamma_\pi]$ and yield a shift along $[\gamma_\pi]$ in $\mathbb{D}$ having maximal dilatation not exceeding
	$$O(K^{\beta_1})\log^2 e^{O(K^2)}=O(K^{\beta+4}).$$
\end{proof}

\begin{remark}
	\label{remark:semi-projected}
	The homotopy type $[\gamma_\pi]$ of the ``semi-projected'' path $\gamma_\pi$ defined in Proposition~\ref{prp:K_shifts_imply_distance_bounds} is independent of the choice of the representative in $[\gamma]$. The notion is useful in the following context.
	
	Assume that $\varphi:\mathbb{C}\to\mathbb{C}$ is a homeomorphism equal to identity on $\mathbb{D}_1^\infty$. Then $\left(\varphi(\gamma)\right)_\pi=\varphi(\gamma_\pi)$ and the maximal dilatation of the shift along $[\varphi(\gamma_\pi)]$ might be estimated using solely the maximal dilatation of the shift along $[\varphi(\gamma)]$ (which will be recovered in a simple way on each step of Thurston's iteration). Moreover, if $\varphi$ is an ``almost identity'' on $\mathbb{D}_q^\infty$ in a sense that for every $z\in\mathbb{D}_q^\infty$, $\|\varphi(z)-z\|_{\cyl}<\delta$ for some small enough $\delta>0$, then $[\varphi(\gamma_\pi)]$ differs from $[(\varphi(\gamma))_\pi]$ solely by concatenation with a short straight interval. Therefore, the corresponding dilatations are commensurable:
	$$K_{\mathbb{C}\setminus\varphi(X\cup Y)}\left(\varphi\circ\gamma(0),[(\varphi\circ\gamma)_\pi]\right)=O\left(K_{\mathbb{C}\setminus\varphi(X\cup Y)}(\varphi\circ\gamma(0),\varphi_*[\gamma_\pi])\right),$$
	where $O(.)$ depends on $p,q$ and $\delta$.
\end{remark}

Next two lemmas provide bounds on how $K_U$ change after application of a \qc\ map and after lift under a branched covering map.

\begin{lmm}[Quasiconformal change of coordinates]
	\label{lmm:qc_change_of_coordinates}
	Let $U, W\subset\mathbb{C}$ be open sets, $\gamma:[0,1]\to U$ be a path and $\varphi:U\to W$ be a $K$-\qc\ map. Then
	$$K_W\left(\varphi\circ\gamma(0),[\varphi\circ\gamma]\right)\leq K^2 K_U\left(\gamma(0),[\gamma]\right).$$
\end{lmm}
\begin{proof}
	For a shift $\psi$ along $[\gamma]$ in $U$, we can consider the induced shift $\psi'$ along $[\varphi\circ\gamma]$ in $W$. It is given by the formula $\psi':=\varphi\circ\psi\circ\varphi^{-1}$. The estimate on $K_W$ follows.
\end{proof}

\begin{lmm}[Lifting properties]
	\label{lmm:K_U_for_lifts}
	Let $U$ be a Riemann domain, $f:\tilde{U}\to U$ be a holomorphic branched covering of degree $d<\infty$ with the only branching value $x\in U$, $\gamma:[0,1]\to U$ be a path starting at $x$ and $\tilde{\gamma}:[0,1]\to \tilde{U}$ be one of its (homeomorphic) lifts under $f$ starting at $\tilde{x}\in f^{-1}(x)$. Then $K_{\tilde{U}}(\tilde{x},[\tilde{\gamma}])\leq d K_U(x,[\gamma])$.
\end{lmm}
\begin{proof}
	After conformal change of coordinates we may assume that $\tilde{U}=U=\mathbb{D}$, $\tilde{x}=x=0$ and $f=z^d$ (this will not change the values of $K_U$ and $K_{\tilde{U}}$). Then for $t\in[0,1]$, we have $\abs{\tilde{\gamma}(t)}=\abs{\gamma(t)}^{1/d}$.
	
	It is enough to show for every $t\in[0,1]$ that if $\psi$ is a $K$-\qc\ shift mapping $0$ to $\gamma(t)$, then there is a $dK$-\qc\ shift $\tilde{\psi}$ mapping $0$ to $\tilde{\gamma}(t)$. Without loss of generality we may assume that $\arg\gamma(t)=\arg\tilde{\gamma}(t)$. Then the map $\tilde{\psi}=\abs{\psi}^{1/d}e^{\arg\psi}$ is the required shift.	
\end{proof}

Lemma~\ref{lmm:K_U_for_lifts} is suited for the situations where we lift a path containing a critical point. However, it is applicable only when $U$ is a Riemann domain. The last definition in this subsection allows usage of Lemma~\ref{lmm:K_U_for_lifts} in higher generality: we split the path homotopy type into a concatenation of two paths, so that the shift along the first one happens inside of some Riemann domain.

\begin{defn}[$K$-decomposability]
	\label{defn:K_decomposability}
	Let $X\subset\mathbb{C}$ be a closed set and $x\in X$ be an isolated point. For a path $\gamma:[0,1]\to\mathbb{C}\setminus X\cup\{x\}$ such that $\gamma(0)=x$ and $x\notin\gamma((0,1])$, we say that $\gamma$ is \emph{$K$-decomposable} for $X$ if there exists $\tau\in(0,1]$ and a Riemann domain $D\subset\mathbb{C}\setminus X\cup\{x\}$ containing $\gamma([0,\tau))$ such that
	$$K_D\left(x,[\gamma|_{(0,\tau]}]\right)K_{\mathbb{C}\setminus X}\left(\gamma(\tau),[\gamma|_{(\tau,1]}]\right)<K.$$
	
	The homotopy class $[\gamma]$ is called \emph{$K$-decomposable} for $\mathcal{O}$ if it contains a $K$-decomposable path.
\end{defn}

It is clear that $K$-decomposability yields a shift (along $[\gamma]$) of a special form: first, the puncture is shifted relative to the boundary of a Riemann domain, afterwards it is shifted relative to $X$ (note that $x$ is also included as a puncture for the second shift).

\subsection{$(K,\delta)$-regularity of tracts}
\label{subsec:tracts_regularity}

This subsection discusses the notion of the $(K,\delta)$-regularity, which serves in Theorem~\ref{thm:invariant_structure} as a glue between two different ways to storage the information describing homotopy type: that a homeomorphism is ``almost'' identity near $\infty$ is encoded in terms of cylindrical metric, while the homotopy and moduli type information for marked points near the origin is encoded using ``fat spiders'' (see next subsection).

\begin{defn}[$(K,\delta)$-regularity]
	\label{defn:K_regularity}
	For an entire function $f\in\mathcal{B}$, consider a triple $(\hat{T}\supset T, X)$ of tracts $\hat{T}\supset T$ such that $f(\hat{T})=\mathbb{C}\setminus\overline{\mathbb{D}}_{\hat{\rho}}$, $f(T)=\mathbb{C}\setminus\overline{\mathbb{D}}_\rho$ for some $0<\hat{\rho}<\rho$ and of a set $X\subset \mathbb{C}$. 
	
	For $K>1,\delta>0$, we say that the triple $(\hat{T}\supset T, X)$ is \emph{$(K,\delta)$-regular} if the following conditions are satisfied.
	\begin{enumerate}
		\item $$K_{\hat{T}\cap\mathbb{D}_{\rho e^{\delta}}\setminus X}\left(\partial T\cap \mathbb{D}_\rho\gg \partial \mathbb{D}_\rho\right)\leq K,$$
		\item for every $x\in X\cap T\cap\mathbb{D}_{\rho}$, there exists a Riemann domain $U_x\subset\hat{T}\cap\mathbb{D}_{\rho e^{\delta}}\setminus X\cup\{x\}$ such that
		$$K_{U_x}(\{x\}\gg \partial \mathbb{D}_\rho)\leq K.$$
	\end{enumerate}
	
	The triple $(\hat{\rho},\rho, X)$ is called \emph{$(K,\delta)$-regular} if $(1)-(2)$ hold for every pair of tracts $\hat{T}\supset T$.
\end{defn}

Condition $(1)$ of Definition~\ref{defn:K_regularity} means that we can shift any point of the boundary of $T$ which lies in $\mathbb{D}_\rho$ to the circle $\partial\mathbb{D}_\rho$ with dilatation at most $K$ and relative to $X$, $\partial\hat{T}$ and $\partial\mathbb{D}_{\rho e^{\delta}}$. Condition $(2)$ means the same for points of the set $X$ except that the corresponding point $x\in X$ has to be ``unmarked'' and the shift ``lives'' on a Riemann domain not containing other marked points.

\begin{remark}
	One could replace condition $(2)$ by a simpler version: 	 $\forall x\in X\cap T$,
	$$K_{\hat{T}\cap\mathbb{D}_{\rho e^{\delta}}\setminus X}(\{x\}\gg \partial \mathbb{D}_\rho)\leq K.$$
	However, this would make it more difficult to treat such singular portraits when a singular value is not the first point on a marked orbit.
\end{remark}

Next proposition holds for all functions having finite (and some ``modest'' infinite) order.

\begin{prp}[$\log$-regularity for finite order]
	\label{prp:log_regularity_finite_order}
	Fix constants $\varepsilon>1$ and $n\in\mathbb{N}$. Let $f\in\mathcal{B}$ satisfy the inequality
	$$\max_{\partial\mathbb{D}_r}\abs{f(z)}<e^{e^{\log^d r}}$$
	for some constant $d>0$ and all $r>0$ big enough. Consider two radii $\hat{\rho}<\rho$, two tracts $\hat{T}\supset T$ such that $f(\hat{T})=\mathbb{C}\setminus\overline{\mathbb{D}}_{\hat{\rho}}, f(T)=\mathbb{C}\setminus\overline{\mathbb{D}}_{\rho}$ and a finite set $X=\{x_1,x_2,\dots,x_m\}\subset T\cap \mathbb{D}_{\rho}$.
	
	Assume further that $\log(\rho/\hat{\rho})>\varepsilon$ and for the set $F\circ\log|_T X$ (where $F$ is the logarithmic transform of $f$) holds
	\begin{enumerate}
		\item $F\circ\log|_T X\subset\mathbb{H}_{\log\rho +\varepsilon}$,
		\item all points in $F\circ\log|_T X$ have Euclidean distance at least $\varepsilon$ from each other.
	\end{enumerate}	
	
	If $\hat{\rho}$ is big enough, then $(\hat{T}\supset T, X)$ is $(K,\varepsilon)$-regular with $K=O(\log^{2d(m+1)}\rho)$ and $O(.)$ depending only on $\varepsilon$ and $m$.
\end{prp}
\begin{proof}
	After choosing $\hat{\rho}$ bigger we might assume that $\max_{z\in\partial\mathbb{D}_\rho}\abs{f(z)}<e^{e^{\log^d\rho}}$. In the logarithmic coordinates this implies $\Re F(z)<e^{\log^d\rho}$ on the vertical line $\Re z=\log\rho$.
	
	Instead of working with $\hat{T},T$ and the set $X$, we switch to their conformal images under $F\circ\log|_{\hat{T}}$. First, we focus on the statement for points in $\partial T$. Let us show that there exist $K$-\qc\ shifts $\psi=\psi_w$ of every point $w$ of the line $\Re z=\log\rho$ to the line $\Re z=e^{\log^d\rho}$ in $\mathbb{H}_{\log\hat{\rho}}\setminus F\circ\log|_T X$ such that $K=O(\log^{2m+4}\rho)$.

	This can be easily done by a most $m+2$ application of Corollary~\ref{cor:moving_inside_round_disk_bounds} together with intermediate ``corrections''. More precisely, let $Y=\{y_1,\dots,y_m\}$ be the set $F\circ\log|_T X$ ordered by increase of the real parts. Assume first that $\abs{\Re (y_{j+1}-y_j)}\geq 4$ and let $v$ be a point on the vertical line $\Re z =\Re y_{j}+1$. Consider a round disk $D_j$ with center at $\Re(y_{j}+y_{j+1})/2+i\Im v$ with radius $r_j=\Re(y_{j+1}-y_{j})/2$. The closed disk $\overline{D'_j}$ with the same center and radius $r_j-1$ intersects the vertical line $\Re z =\Re y_{j}+1$ at $v$ and the line $\Re z=\Re y_{j+1}-1$ at some other point $v_1$. Due to Corollary~\ref{cor:moving_inside_round_disk_bounds} we can map $v$ to $v_1$ via a \qc\ map equal to identity outside of $D_j$ and having maximal dilatation
	$$O\left(\log^2\frac{1}{r_j}\right)=O\left(\log^2e^{\log^d\rho}\right)=O(\log^{2d}\rho).$$
	Thus, any point of the line $\Re z =\Re y_{j}+1$ can be mapped to the line $\Re z=\Re y_{j+1}-1$ in such a way.
	
	If there is a cluster of $\{y_i,y_{i+1},\dots,y_{i+k}\}$ with consecutive differences of real parts less than $4$, we can move through such cluster with some bounded dilatation depending only on $\varepsilon$ and $m$ which are initially fixed. In any case, there is at most $m$ such clusters.
	
	Composing all consecutive maps ``between'' and ``through'' the clusters, we conclude that the maximal dilatation of the resulting \qc\ map is at most $O(\log^{2d(m+1)}\rho)$. 
	
	The real part of points in the set $F\circ\log|_T (T\cap \mathbb{D}_{\rho})$ is generally smaller than $e^{\log^d\rho}$. In order to obtain shifts equal to identity outside of $F\circ\log|_T (T\cap \mathbb{D}_{e^\varepsilon\rho})$, we simply need to ``truncate'' the construction above and recall that $F$ is expanding on tracts if $\hat{\rho}$ is big enough.
	
	The proof for points $x\in X$ is identical, except that one replaces $m$ by $m-1$. Possibility of the property that the corresponding shift is equal to identity outside of some Riemann domain $U_x$ is evident from the construction.
\end{proof}

\subsection{Fat spiders}
\label{subsec:spiders}
In this subsection we introduce a special description for certain type of points in (finite-dimensional) \tei\ spaces. 

\begin{defn}[Fat spider $S(A,\mathcal{L}, Y)$]
	\label{defn:fat_spider}
	Let $A\subset\mathbb{C}$ be an open annulus of finite modulus with $B\ni\infty$ (``body'') and $G$ (``ground'') being the unbounded and bounded components of $\hat{\mathbb{C}}\setminus A$, respectively, and $Y\ni\infty$ be a subset of $B\cup\{\infty\}$. Also, let $L_{i}:[0,1]\to\mathbb{C}, i=\overline{1,n}$ be continuous paths such that all $L_i(0)$'s are distinct and for every $i$: 
	\begin{enumerate}
		\item $L_i(0)\in G$,
		\item range of $L_i|_{(0,1)}$ is $\mathbb{C}\setminus\left(\overline{Y}\cup\{L_1(0),\dots,L_{i-1}(0),L_{i+1}(0),\dots,L_n(0)\}\right)$,
		\item $L_i(1)\in B\setminus \overline{Y}$.
	\end{enumerate}
	Then we say that a triple $(A,\{[L_{i}]\}_{i=1}^n,Y)$ is a \emph{fat spider} with \emph{separating annulus} $A$, \emph{removed set} $Y$ and \emph{legs} $\{[L_i]\}_{i=1}^n$ (the homotopy type of the leg $[L_i]$ is defined in $\mathbb{C}\setminus\left(\overline{Y}\cup\{L_1(0),\dots,L_{i-1}(0),L_{i+1}(0),\dots,L_n(0)\}\right)$).
\end{defn}

\begin{figure}[h]
	\includegraphics[width=\textwidth-15em]{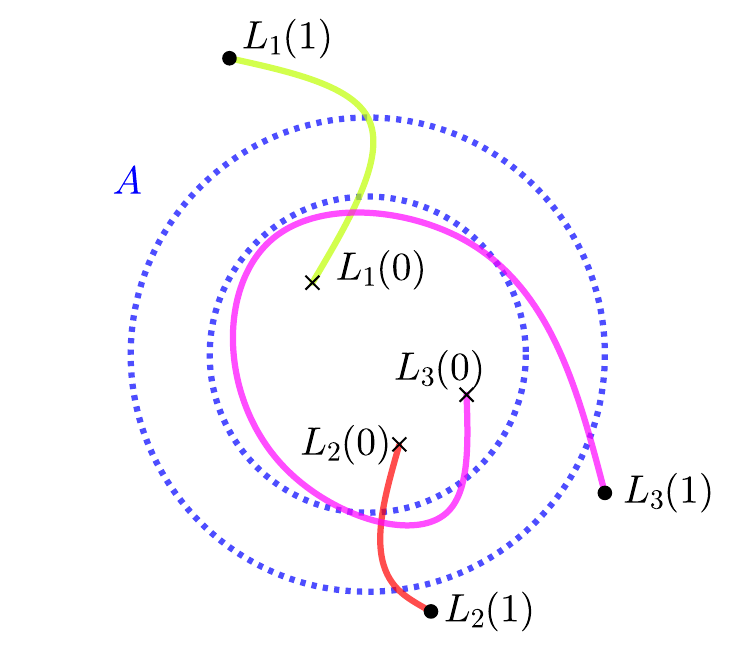}
	\caption{Example of a fat spider on 3 legs. The boundary of a separating annulus is depicted in blue.}
	\label{pic:fat_spider}
\end{figure} 
	
Introduce also the notation $\mathcal{L}:=\{[L_1],\dots,[L_n]\}$, $\mathcal{L}(0):=\{[L_1](0),\dots,[L_n](0)\}$ and $\mathcal{L}(1):=\{[L_1](1),\dots,[L_n](1)\}$. By $K^S_i$ we denote the maximal dilatation of the shifts in $\mathbb{C}\setminus\left(\overline{Y}\cup\mathcal{L}(0)\right)$ mapping $L_i(0)$ along $[L_i]$, i.e.,
	$$K^S_i:=K_{\mathbb{C}\setminus\left(\overline{Y}\cup\mathcal{L}(0)\right)}(L_i(0),[L_i]).$$

\begin{defn}[Fat spider map]
	\label{defn:fat_spider_map}
	Let $S^1(A,\mathcal{L}^1,\{\infty\}),S^2(A,\mathcal{L}^2,\{\infty\})$ be two fat spiders on $n$ legs with the same separating annulus $A$ (and body $B$). We say that a homeomorphism $\varphi:\mathbb{C}\to\mathbb{C}$ is a \emph{fat spider map} $\varphi:S^1\to S^2$ if $\varphi|_B=\id$ and for every $i$, $\varphi_*[L_i^1]=[L_i^2]$.	  
\end{defn}

\begin{figure}[h]
	\includegraphics[width=\textwidth+4em]{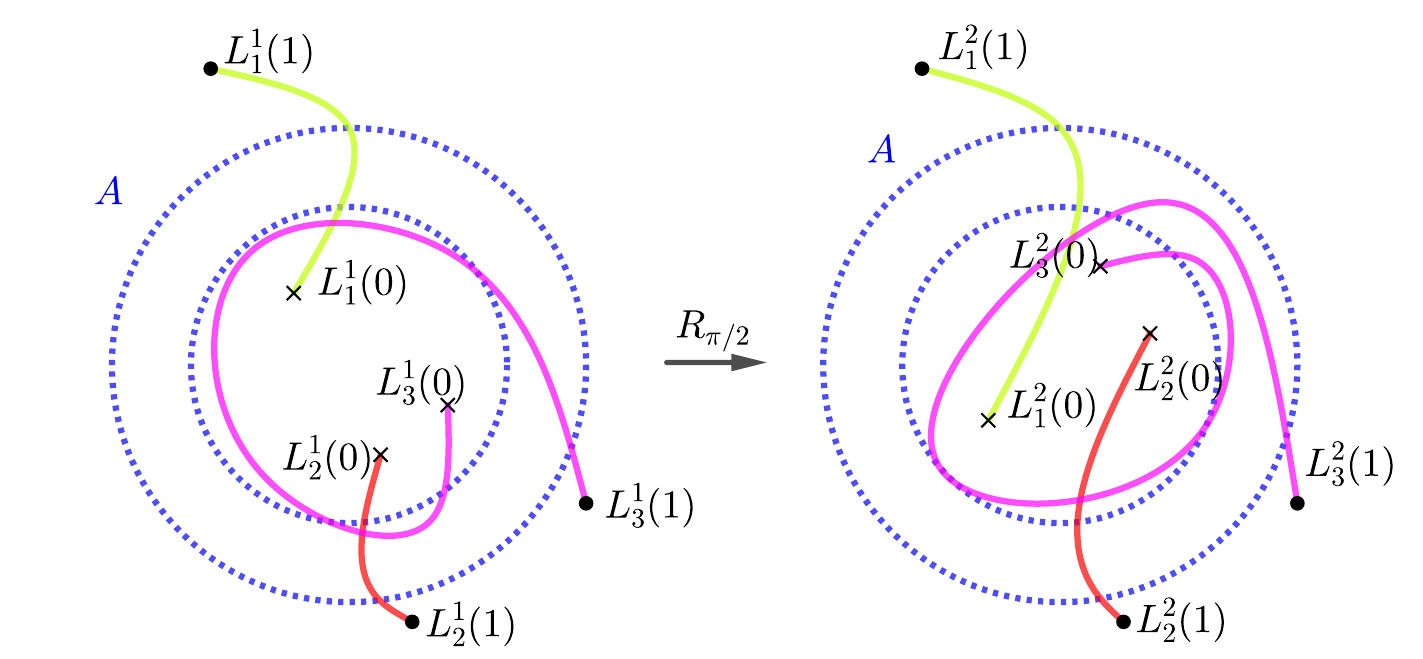}
	\caption{Fat spider map equivalent to a counterclockwise twist by $\pi/2$. Note that the map respects the homotopy type of legs.}
	\label{pic:fat_spider_map}
\end{figure} 

The fat-spider-language helps to formulate compactly the following proposition which will play one of the key roles in the construction of the invariant subset of the \tei\ space later. 

\begin{prp}[\tei\ metric for fat spider maps]
	\label{prp:teich_metric_fat_spider_map}
	Fix a constant $M>0$ and a natural number $n>2$. Let $S^1(A,\mathcal{L}^1,\{\infty\}),S^2(A,\mathcal{L}^2,\{\infty\})$ be two fat spiders on $n$ legs such that $A$ is a round annulus around the origin, $\mod A\geq M$ and $K_i^j:=K_i^{S_j}<K$ for all pairs $i,j$.
	
	If $\varphi:S^1\to S^2$ is a fat spider map, then $\varphi$ is isotopic relative to $\mathcal{L}^1(0)\cup B$ to a $K'$-\qc\ map $\varphi':\mathbb{C}\to\mathbb{C}$ with
	$$K'<\left(C K\right)^{\nu n},$$
	where $C>0$ depends only on $M$ and $\nu>0$ is a universal constant.
\end{prp}

\begin{remark}
	The cases $n=1,2$ are rather special. For $n=2$ the proposition is false, but luckily this will not cause any difficulties later. For $n=1$ one can immediately see that $K'=O(1)$.
\end{remark}

\begin{proof}
	Without loss of generality we might assume that the outer radius of $A$ is equal to $1$. Otherwise apply a linear change of coordinates.
	
	If $\psi_i^j,j={1,2}$ is a $K$-shift of $L_i^j(0)$ along $[L_i^j]$ in $\mathbb{C}\setminus\mathcal{L}^j(0)$, then due to Proposition~\ref{prp:K_shifts_imply_distance_bounds}, there exists a $K_1$-shift $\chi_i^j$ of $L_i^j(0)$ along $[L_i^j]$ in $\mathbb{D}_{2\abs{L_i^j(1)}}\setminus\mathcal{L}^j(0)$ such that $\chi_i^j$ is isotopic relative to $\mathcal{L}^j(0)$ to $\psi_i^j$ and $\chi_i^j=\id$ on $\mathbb{C}\setminus\mathbb{D}_{2\abs{L_i^j(1)}}$, where $K_1=O(K^\beta)$ for a universal constant $\beta>0$.
	
	We use inductive argument. Consider the homeomorphism
	$$\chi_n^2\circ\varphi\circ(\chi_n^1)^{-1}.$$
	Note that it is isotopic relative to $\left(\mathbb{C}\setminus\mathbb{D}_{2\abs{L_n^1(1)}}\right)\cup\left(\mathcal{L}^1(0)\setminus\{L_n^1(0)\}\right)\cup\{L_n^1(1)\}$ to a homeomorphism $\varphi_1:\mathbb{C}\to\mathbb{C}$ equal to identity on $B$. This implies that the maximal dilatation $K_\varphi$ induced by $\varphi$ is bounded by the maximal dilatation $K_{\varphi_1}$ induced by $\varphi_1$ times the maximal dilatations of $\chi_n^j$, i.e., $K_\varphi\leq (K_1)^2 K_{\varphi_1}$. 
	
	At the same time, $\varphi_1$ is isotopic to $\varphi$ relative to $B\cup\mathcal{L}^1(0)\setminus\{L_n^1(0)\}$. We repeat this procedure for $\varphi_1$ (but without the legs $[L_n^1],[L_n^2]$) and proceed inductively. After noticing that if $n=1$, the corresponding induced maximal dilatation is equal to $O(1)$, we see that $K_\varphi=O\left((K_1)^{2(n-1)}\right)$.
	
	Since all isotopies are relative to $\max_{1<i\leq n}\{\abs{L_i^1(1)}\}$ and by the proof of Proposition~\ref{prp:K_shifts_imply_distance_bounds}, $\abs{L_i^1(1)}=e^{O(K^2)}$, application of Lemma~\ref{lmm:conformal_neighbourhood_expand} provides us the desired map $\varphi'$ with maximal dilatation
	$$K'<O(K_1)^{2(n-1)}O(K^4)<(C K)^{2(n-1)\beta+4}=(C K)^{\nu n}.$$
\end{proof}

\section{Invariant structure}
\label{sec:invariant_structure}
In this section after a few preparational definitions we finally state and prove the main structural result: Theorem~\ref{thm:invariant_structure}.

Let $f:\mathbb{C}\to\mathbb{C}$ be a quasiregular function and $\{a_{ij}\}$ with $i=1,...,m$ and $j=1,2,...$ be $m$ marked orbits.

\begin{defn}[$N(\rho), N_i(\rho)$]	
	By $N(\rho)$ we denote the number of pairs $i,k$ such that $a_{ik}\in\mathbb{D}_\rho$ (in particular, $N(\rho)$ can be equal to $\infty$).
	
	Analogously, by $N_i(\rho)$ we denote the number of $k$'s such that $a_{ik}\in\mathbb{D}_\rho$.
\end{defn}		

Let us return to the context of Thurston's iteration, i.e., consider a captured quasiregular map $f=\lambda\circ f_0$ as constructed in Subsection~\ref{subsec:iteration_setup} with $m$ marked orbits containing all singular values.

We will need the following rather lengthy definition.

\begin{defn}[Separating structure]
	\label{defn:separating_structure}
	Let $f_0$ be a transcendental entire function of finite type and $U\supset\SV(f_0)$ be a bounded domain.
	By a \emph{separating structure} $\mathbb{S}[\rho, q, K_0, K_1, \varepsilon]$ for $f_0$ and $U$, where $\rho>0$, $0<q<1/2$, $K_0,K_1\geq 1$, $0<\varepsilon<1$, we understand the following list of interdependent objects and conditions on them:
	\begin{enumerate}
		\item a $K_0$-\qc\ map $\lambda$ so that $\lambda|_{\mathbb{C}\setminus\mathbb{D}_{q\rho}}=\id$;
		\item $m$ marked orbits $\mathcal{O}:=\{a_{ij}\}$ of the quasiregular map $f=\lambda\circ f_0$ such that
		\begin{enumerate}
			\item $\SV(f)\subset\mathcal{O}\cap\mathbb{D}_{q\rho}$,
			\item $\mathcal{O}\cap \lambda(U)=\SV(f)$,
			\item $\mathcal{O}\cap \mathbb{A}_{q\rho, \rho e^{\varepsilon}}=\emptyset$,
			\item $\mathcal{O}\setminus\mathbb{D}_{\rho}$ is forward invariant,
			\item $N(\rho)<\infty$;
		\end{enumerate}
		\item the triple $(\rho/2,\rho,\mathcal{O})$ is $(K_1,\varepsilon/2)$-regular;
		\item if $N_i(\rho), N_k(\rho)>0, i\neq k$ and $a_{iN_i},a_{kN_k}$ belong to the same asymptotic tract $f^{-1}(\mathbb{C}\setminus\mathbb{D}_{\rho})$, then $d_{\cyl}(a_{i(N_i+1)},a_{k(N_k+1)})>\varepsilon$.
	\end{enumerate}
\end{defn}

Clearly, a separating structure is not defined uniquely by its parameters but rather describes an ``environment'' to work in.

Note that the definition of a separating structure forbids marked cycles inside of $\mathbb{D}_\rho$, but allows them in general.

The coefficient $1/2$ at $\rho$ for the triple $(\rho/2,\rho,\mathcal{O})$ can be replaced by any other positive constant smaller than $1$, the only condition is that it has to remain bigger than $q$. The conclusions of all subsequent theorems will remain valid.

Denote by $\omega_{ij}$ the local degree of $f$ at $a_{ij}$ (hence $\omega_{ij}=1$ unless $a_{ij}$ is a critical point), denote
$$\Omega_{ij}(\beta):=\prod_{k=j}^\infty\omega_{ik}^{\beta^{k-j+1}}$$
and let $\Omega(\beta):=\max_{ij}\Omega_{ij}(\beta)$.

\begin{lmm}[Initial fat spider]
	\label{lmm:initial_fat_spider}
	Let $\mathbb{S}[\rho, q, K_0, K_1, \varepsilon]$ be a separating structure for $f_0$ and let $N(\rho)\neq 2$. Then there exists a fat spider $S(\mathbb{A}_{q\rho,\rho},\{[L_{ij}]\}_{j\leq N_i},\mathcal{O})$ with $N(\rho)$ legs $[L_{ij}]$ for $j\leq N_i(\rho)$ such that
	\begin{enumerate}
		\item $[L_{ij}]$ joins $[L_{ij}](0)=a_{ij}\in \mathbb{D}_{q\rho}$ to a point $[L_{ij}](1)\in\partial\mathbb{D}_{\rho}$,
		\item $[L_{ij}]$ is $K_{ij}$-decomposable for $\mathcal{O}$ where 
		$$K_{ij}<(B K_1K_0^2)^{\beta^{N_i-j}}\Omega_{ij}$$
		where $\beta>1$ and $B>1$ are universal constants.
	\end{enumerate} 
\end{lmm}
\begin{proof}	
	For $N(\rho)=1$ the statement is obvious so let us assume that $N(\rho)>2$.
	
	The proof uses inductive argument. Consider a point $a_{i N_i}\in T\cap \mathbb{D}_\rho$, where $T=f^{-1}(\mathbb{C}\setminus\mathbb{D}_{\rho})$ is an asymptotic tract. From the $(K_1,\varepsilon/2)$-regularity, it follows that there exists a leg $[L_{iN_i}]$ from $a_{iN_i}$ to $\partial\mathbb{D}_{\rho}$ which is $K_1$-decomposable for $\mathcal{O}$. That is, $K_{iN_i}\leq K_1$.
	
	To construct $[L_{i(N_i-1)}]$, consider the pre-image $[\gamma]$ of $[L_{iN_i}]$ under $f=\lambda\circ f_0$ which starts at $a_{i(N_i-1)}$. By Lemma~\ref{lmm:K_U_for_lifts} and Lemma~\ref{lmm:qc_change_of_coordinates}, it is $K_{\gamma}$-decomposable where
	$$K_{\gamma}=K_{iN_i}K_0^2\omega_{i(N_i-1)}.$$
	Generally, we cannot assign $[L_{iN_i}]:=[\gamma]$ because the endpoint $[\gamma](1)$ might not belong to $\partial\mathbb{D}_{\rho}$, hence we will extend $[\gamma]$. There are two different types of prolongation of $[\gamma]$ depending on the position of $[\gamma](1)$.
	
	\begin{enumerate}
		\item ($[\gamma](1)\notin\mathbb{D}_\rho$) Assign $[L_{i(N_i-1)}]$ to be equal to $[\gamma_\pi]$ where $\gamma_\pi$ is the ``semi-projection'' in the circle $\partial\mathbb{D}_\rho$ similarly as defined in Proposition~\ref{prp:K_shifts_imply_distance_bounds}. The endpoint of $[\gamma_\pi]$ belongs to $\partial\mathbb{D}_\rho$ and, due to Proposition~\ref{prp:K_shifts_imply_distance_bounds} (note that there is at least one leg $[L_{iN_i}]$),
		$$K_{\gamma_\pi}<C_1K_{\gamma}^{\beta_1+4}$$
		for some universal constants $C_1$ and $\beta_1$ (universality is due to inequalities $q<1/2$ and $\varepsilon<1$). Note that $[L_{i(N_i-1)}]$ can be written as the concatenation of the three homotopy classes in the provided order:
		$[\gamma]$, reversed $[\gamma]$ and $[\gamma_\pi]$. Thus, $[L_{i(N_i-1)}]$ is $K_{i(N_i-1)}$-decomposable with
		$$K_{i(N_i-1)}<K_{\gamma}K_{\gamma} C_1K_{\gamma}^{\beta_1+4}=C_1(K_{iN_i}K_0^2\omega_{i(N_i-1)})^{\beta_1+6}.$$
		The decomposability takes place because we have the concatenation starting with a decomposable homotopy type.
		\item ($[\gamma](1)\in\mathbb{D}_\rho$) Note that $[\gamma](1)\in\partial T$ where $T=f^{-1}(\mathbb{C}\setminus\overline{\mathbb{D}}_{\rho})$ is an asymptotic tract. Exactly as in the case above, we see that $[\gamma]$ is $K_{\gamma}$-decomposable for $\mathcal{O}$. 	Let $[L_{i(N_i-1)}]$ be the concatenation of $[\gamma]$ with a homotopy class of paths from $[\gamma](1)$ to $\partial\mathbb{D}_{\rho}$ existing due to $(K_1,\varepsilon/2)$-regularity. Then, since $\mathcal{O}\cap \mathbb{A}_{q\rho, \rho e^{\varepsilon}}=\emptyset$, $[L_{i(N_i-1)}]$ is $K_{i(N_i-1)}$-decomposable with 
		$$K_{i(N_i-1)}\leq K_{[\gamma]}K_1=K_{iN_i}K_1K_0^2\omega_{i(N_i-1)}.$$
	\end{enumerate}
	
	Proceeding by induction and ``unifying'' two cases, one can show existence of all $[L_{ij}]$ such that
	$$K_{iN_i}<K_1<BK_1K_0^2$$
	and
	$$K_{ij}< K_{i(j+1)}^{\beta_1+6} (BK_1K_0^2\omega_{ij})^{\beta_1+6}\leq K_{i(j+1)}^{2(\beta_1+6)}\omega_{ij}^{2(\beta_1+6)}=K_{i(j+1)}^\beta\omega_{ij}^\beta.$$
	The claim of the lemma follows.
\end{proof}

\begin{defn}[Standard spiders]
	\label{defn:standard_spiders}
	Let $\mathbb{S}[\rho, q, K_0, K_1, \varepsilon]$ be a separating structure for a function $f_0$ (i.e., with some fixed choice of parameters). Denote by $\mathcal{S}_0$ the set of fat spiders satisfying conditions of Lemma~\ref{lmm:initial_fat_spider}. We call it \emph{the set of standard (fat) spiders associated to the separating structure  $\mathbb{S}[\rho, q, K_0, K_1, \varepsilon]$}.	
\end{defn}

The main benefit of Definition~\ref{defn:standard_spiders} is that one can define a dynamically meaningful pull-back (via $f$) keeping the set $\mathcal{S}_0$ invariant. In other words, there is a procedure (by a slight abuse of terminology called \emph{pull-back}) producing from every standard spider $S\in\mathcal{S}_0$ a new standard spider $\tilde{S}\in\mathcal{S}_0$ by literally repeating the algorithm in the proof of Lemma~\ref{lmm:initial_fat_spider}: we take a pre-image of a leg under $f$ and choose its prolongation. From the proof it is clear that all bounds are invariant.

Note that this procedure \emph{does not} define the spider $\tilde{S}$ uniquely. However, instead of artificially making some precise choice, we should rather think that the pull-back gives as its output \emph{some} spider $\tilde{S}\in\mathcal{S}_0$, a particular choice is irrelevant for us.

We are finally ready to describe the situation in which existence of a separating structure with ``good'' bounds implies existence of a certain invariant structure associated to Thurston's pull-back map.

\begin{thm}[Invariant structure]
	\label{thm:invariant_structure}
	Let $f_0$ be a finite type entire function, $D$ be a union of Riemann domains with pairwise disjoint closures and each containing exactly one singular value and $U\supset\overline{D}$ be a bounded domain. Fix a real number $0<\varepsilon<1$.
	
	There is such universal constant $\nu>1$ and (non-universal) constants $\rho_0(f_0, U, \varepsilon)>0$, $q_0(\varepsilon)<1/2$, $\Delta(U,D)>0$ that existence of a separating structure $\mathbb{S}[\rho, q, K_0, K_1, \varepsilon]$ for $f_0$ and $U$ satisfying inequalities $\rho\geq\rho_0$, $q\leq q_0$ and
	\begin{equation}
	\label{eqn:dilatation_per_area_condition_final}
	I_q(\rho,D)\left(K_1K_0\max\limits_i\prod_{j=1}^{\infty}\omega_{ij}\right)^{\nu^{N(\rho)}}<\Delta
	\end{equation}
	implies that there is a nonempty set $\mathcal{I}\subset\hat{\mathcal{T}}_{\mathcal{O}}$ of \tei\ equivalence classes $[\varphi]$ of (topological) homeomorphisms such that
	\begin{enumerate}
		\item $\mathcal{I}$ is $\sigma$-invariant;
		\item the projection of $\mathcal{I}$ to $\mathcal{T}_{\{a_{ij}: j\leq N_i(\rho)+1\}}$ is a bounded set;
		\item every equivalence class $[\varphi]\in\mathcal{I}$ contains a homeomorphism $\varphi$ such that for every $z\in\mathbb{D}_\rho^\infty$, we have $d_{\cyl}(\varphi(z),z)<\varepsilon/4$.
	\end{enumerate}

\end{thm}
\begin{proof}
	It will be more convenient to wor with a slightly more general version of the second paragraph of the theorem:
	
	\emph{...There are such universal constants $C>1,\nu>1,\beta>1$ and (non-universal) constants $\rho_0(f_0, U, \varepsilon)>0$, $q_0(\varepsilon)<1/2$, $\Delta(U,D)>0$ that existence of a separating structure $\mathbb{S}[\rho, q, K_0, K, \varepsilon]$ for $f_0$ and $U$ satisfying inequalities $\rho\geq\rho_0$, $q\leq q_0$ and
		\begin{equation}
		\label{eqn:dilatation_per_area_condition}
		\left(CK_1K_0\right)^{\nu^{N(\rho)}}\Omega^{\nu N(\rho)}(\beta)I_q(\rho,D)<\Delta
		\end{equation}
	implies that there is...}

	It is clear that the former version follows from the latter after adjustments of constants (note that $K_1\to\infty$ as $\rho\to\infty$).
	
	First, without loss of generality we might assume that $N=N(\rho)\neq 2$. If $N=2$, we just add an additional marked point (for example, non-marked pre-image of some $a_{i(N_i+1)}$).
	
	Let $\alpha>1$ be some universal constant and $0<\delta<\varepsilon/4$ be a parameter depending only on $\varepsilon$, their precise values we determine later. We will require a more elaborate description for $\mathcal{I}$ than in the statement of the theorem. It is defined as the set of equivalence classes $[\varphi]$ containing a homeomorphism $\varphi:\mathbb{C}\to\mathbb{C}$ satisfying the following list of conditions:	
	\begin{enumerate}
		\item $\varphi|_{\partial\mathbb{D}_\rho}=\id$,
		\item for $z\in\mathbb{D}_\rho^\infty$, $d_{\cyl}(\varphi(z),z)<\delta$,
		\item there is a standard fat spider $S$ with $K_{ij}$-decomposable (for $\mathcal{O}$) legs $[L_{ij}]$ so that $S':=\varphi(S)$ is a fat spider with the separating annulus $\mathbb{A}_{\alpha q\rho,\rho/2}$ and $K'_{ij}$-decomposable (for $\varphi(\mathcal{O})$) legs $[L'_{ij}]=\varphi_*[L_{ij}]$ so that
		$$K'_{ij}<(8BK_1K_0^2)^{\beta^{N_i-j}}\Omega_{ij}(\beta)$$
		where $B$ and $\beta$ are universal constants from the definition of the standard spider.
	\end{enumerate}
	Due to Proposition~\ref{prp:teich_metric_fat_spider_map}, it is clear that for any $\mathcal{I}$ of this form, its projection to $\mathcal{T}_{\{a_{ij}: j\leq N_i(\rho)+1\}}$ is bounded, hence it will suffice to prove $\sigma$-invariance in terms of the second description.
	
	$\mathcal{I}$ is non-empty because identity trivially satisfies the conditions for $\varphi$. In order to prove invariance of $\mathcal{I}$ under $\sigma$ we have to show that if (\ref{eqn:dilatation_per_area_condition}) holds for big enough $\rho$ and small enough $q$, then there is a homeomorphism $\hat{\varphi}\in\sigma[\varphi]$ satisfying the same conditions.
	
	So, let $\varphi$ be a homeomorphism satisfying $(1)-(3)$. Note that from the definition of standard spider and since $N_i\leq N$, for every $a_{ij}\in\mathbb{D}_\rho$, we have
	$$K_{ij}<(BK_1K_0^2)^{\beta^{N_i-j}}\Omega_{ij}(\beta)<(8BK_1K_0^2)^{\beta^N}\Omega(\beta)$$
	and
	$$K'_{ij}<(8BK_1K_0^2)^{\beta^N}\Omega(\beta).$$
	
	In order to work with \qc\ maps let us consider the isotopy class of $\varphi$ relative to $\mathcal{O}\cap\mathbb{D}_\rho\cup\partial\mathbb{D}_\rho$. It contains a homeomorphism $\varphi_1$ equal to identity on $\mathbb{D}_{\rho/2}^\infty$. The fat spiders and the homotopy types of their legs project correspondingly and the maximal dilatation coefficients can only decrease, hence we keep denoting them $S$ and $S'$. Assuming that $\mod\mathbb{A}_{\alpha q\rho,\rho/2}>\log 2$ (i.e., $q\alpha<1/4$) and applying Proposition~\ref{prp:teich_metric_fat_spider_map} to the spider map $\varphi_1$ between fat spiders $S$ (with a new smaller separating annulus $\mathbb{A}_{\alpha q\rho,\rho/2}$ rather than the default $\mathbb{A}_{q\rho,\rho}$) and $S'$, we obtain a $K'$-\qc\ map $\varphi_2$ isotopic to $\varphi$ relative to $\mathcal{O}\cap\mathbb{D}_\rho\cup\partial\mathbb{D}_\rho$ such that $\varphi_2|_{\mathbb{D}_{\rho/2}^\infty}=\id$ and 
	\begin{equation}
		\label{eq:main_thm_K'}
		K'<\left(C_1\left(8BK_1K_0^2\right)^{\beta^N}\Omega(\beta)\right)^{\nu_1 N}<\left(CK_1K_0^2\right)^{\nu_1\beta^NN}\Omega^{\nu_1 N}(\beta),
	\end{equation}
	where $C_1,C$ and $\nu_1$ are universal constants (the constant $C_1$ from Proposition~\ref{prp:teich_metric_fat_spider_map} is now universal because the modulus of the separating annulus is bounded from below by $\log 2$). The constant $C$ is defined by the inequality.
	
	Further, if $\rho$ is big enough, then due to Lemma~\ref{lmm:conformal_neighbourhood}, $\varphi'\circ\lambda$ can be isotoped relative to $\lambda^{-1}(\mathcal{O})\cap\mathbb{D}_{q\rho}\cup\mathbb{D}_{\rho/2}^\infty$ to a $C_3 K_0^2 K'^2$-\qc\ map $\chi:\mathbb{C}\to\mathbb{C}$ which is conformal on $D$ and equal to identity on $\mathbb{D}_{\rho/2}^\infty$, where the constant $C_3$ depends only on $U$. Therefore, the maximal dilatation of $\chi$ satisfies
	$$C_3 K_0^2 K'^2<C_3K_0^2\left(CK_1K_0^2\right)^{\nu_1\beta^NN}\Omega^{\nu_1 N}(\beta)<C_3\left(CK_1K_0\right)^{\nu^{N(\rho)}}\Omega^{\nu N}(\beta)$$
	where $\nu>1$ is a universal constant. Note that it is defined at this point.
	
	From inequality~(\ref{eqn:dilatation_per_area_condition}), $C_3 K_0^2K'^2 I_q(\rho,D)<C_3\Delta$. Let $\tilde{\chi}$ be the pull-back of $\chi$ under $f_0$ normalized so that $\tilde{\chi}(0)=0$ and $\tilde{\chi}(z)/z\to 1$ as $z\to\infty$. The normalization is well defined due to \tei--Wittich theorem. Applying Proposition~\ref{prp:distortion of identity} to the map $\tilde{\chi}$ and the round disk $\mathbb{D}_{q\rho}^{\infty}$ centered at $\infty$, we see that there exists $q_0=q_0(\delta)$ such that if $q<q_0$ and $\Delta$ is small enough, then:
	\begin{enumerate}
		\item for every $z\in\mathbb{D}_{\rho/2}^{\infty}$, we have $d_{\cyl}(\tilde{\chi}(z),z)<\delta/3$,
		\item $\tilde{\chi}(\overline{\mathbb{D}}_{q\rho+\varepsilon/4})\subset\mathbb{D}_{\alpha q\rho}$.
	\end{enumerate}
	Note that if we assume that $\Delta<1$, then $\alpha$ can be chosen as a universal constant. So, $\alpha$ is defined at this point.
	
	Denote $g:=\chi\circ f_0\circ\tilde{\chi}^{-1}$. It is clear that $g=\varphi\circ f\circ\tilde{\varphi}^{-1}$ for some $\tilde{\varphi}\in\sigma[\varphi]$ (however, we ``forgot'' the marked points outside of $\mathbb{D}_\rho$). The normalization of $\tilde{\varphi}$ is uniquely determined by the one of $\tilde{\chi}$. Therefore, we can recover a big part of information about the isotopy class of $[\tilde{\varphi}]$ by lifting the isotopy (relative to $\mathcal{O}\cap\mathbb{D}_\rho\cup\partial\mathbb{D}_\rho$) between $\chi$ and $\varphi\circ\lambda$.
	
	By construction, we have $\tilde{\varphi}=\tilde{\chi}$ on $f^{-1}(\partial\mathbb{D}_\rho)$. Let us define the desired map $\hat{\varphi}\in[\tilde{\varphi}]$ by prescribing $\hat{\varphi}:=\tilde{\chi}$ on $f^{-1}(\mathbb{D}_\rho)$ and $\hat{\varphi}:=\tilde{\varphi}$ otherwise. For the moment we ignore the necessary condition that $\hat{\varphi}=\id$ on $\partial\mathbb{D}_\rho$. This flaw will be easily corrected later.
	
	We show that for $z\in\mathbb{D}_\rho^\infty$, $d_{\cyl}(\hat{\varphi}(z),z)<2\delta/3$. Indeed, if $z\in\mathbb{D}_\rho^\infty\cap f^{-1}(\mathbb{D}_\rho)$, then
	$$d_{\cyl}(\hat{\varphi}(z),z)=d_{\cyl}(\tilde{\chi}(z),z)<\delta/3.$$
	Otherwise, if $z\in\mathbb{D}_\rho^\infty\setminus f^{-1}(\mathbb{D}_\rho)$, consider the shortest (cylindrical) geodesic interval between $\chi\circ f_0(z)$ and $\varphi\circ\lambda\circ f_0(z)$. On one hand, its lift under $g$ is a curve joining $\tilde{\chi}(z)$ to $\tilde{\varphi}(z)=\hat{\varphi}(z)$. On the other, the expansion property of $F_0$ on logarithmic tracts shows that the lift of the interval under $\chi\circ f_0$ starting at $z$ and ending at $\tilde{\chi}^{-1}\circ\hat{\varphi}(z)$ has cylindrical length smaller than $\delta/3$ if $\rho$ is big enough. Thus, due to the triangle inequality,
	$$d_{\cyl}(\hat{\varphi}(z),z)\leq
	d_{\cyl}\left(\tilde{\chi}\left(\tilde{\chi}^{-1}\circ\hat{\varphi}(z)\right),\tilde{\chi}^{-1}\circ\hat{\varphi}(z)\right)
	+d_{\cyl}(\tilde{\chi}^{-1}\circ\hat{\varphi}(z),z)<2\delta/3.$$
	
	One can see similarly that $\hat{\varphi}(a_{iN_i})\in\mathbb{D}_{\alpha q\rho}$. Consider the shortest (cylindrical) geodesic interval between $a_{i(N_i+1)}=\chi\circ f_0(a_{iN_i})$ and $\varphi(a_{i(N_i+1)})$. Again, due to the expansion property of $F_0$, its lift under $\chi\circ f_0$ starting at $a_{iN_i}$ and ending at $\tilde{\chi}^{-1}\circ\hat{\varphi}(a_{iN_i})$ has cylindrical length smaller than $\delta$ if $\rho$ is big enough. Thus, $\tilde{\chi}^{-1}\circ\hat{\varphi}(a_{iN_i})\in\mathbb{D}_{q\rho+\delta}\subset\mathbb{D}_{q\rho+\varepsilon/4}$ and property $(2)$ of $\tilde{\chi}$ yields the estimate.
	
	Thus, if $\tilde{S}$ is a fat spider obtained via pull-back of $S$, then its image under $\hat{\varphi}$ is also a fat spider $\tilde{S}'\left(\mathbb{A}_{\alpha q\rho,\rho e^{-\delta}},\{[\tilde{L}'_{ij}]\}_{j\leq N_i},\hat{\varphi}(\mathcal{O}\cap\mathbb{D}_\rho^\infty)\right)$, though generally with a smaller separating annulus. Note that according to our definitions we cannot say that $\tilde{S}'$ is obtained from $S'$ by a pull-back via $g$. However, as it is a homotopic image of $\tilde{S}$, the bounds on the maximal dilatation induced by its legs can be computed with the help of $g$. The following diagram illustrates the relations we have just described.
	\begin{center}
		\begin{tikzcd}
			\tilde{S},[\tilde{L}_{ij}] \arrow[r, "{\hat{\varphi}}"] \arrow[d, "f"]	& \tilde{S}', [\tilde{L}_{ij}'] \arrow[d, "g"] \\
			S, [L_{ij}] \arrow[r, "{\varphi}"] & S', [L_{ij}']
		\end{tikzcd}
	\end{center}
	\vspace{0.5cm}
	
	We will show that the legs $[\tilde{L}_{ij}']=[\hat{\varphi}(L_{ij}')]$ are $\tilde{K}_{ij}'$-decomposable for $\hat{\varphi}(\mathcal{O})$ and compute the upper bounds for $\tilde{K}_{ij}'$. Recall the construction algorithm for pulled-back spider: to obtain $\tilde{S}_{ij}$, first, we take a pre-image of a leg $L_{i(j+1)}$ under $f$ starting at $a_{ij}$ and, second, extend this pre-image in a way depending on whether its endpoint belongs to either $\mathbb{D}_\rho$ or its complement.
	
	However, there is an initial step when $j=N_i$: a leg $[L_{iN_i}]$ exists due to $(K_1,\varepsilon/2)$-regularity of tracts. More precisely, if $\hat{T}\supset T$ are some tracts of $f_0$ corresponding to radii $\rho/2$ and $\rho$, and $a_{iN_i}\in T$, then there exists a Riemann domain $D_i\subset f_0^{-1}(\mathbb{D}_{\rho/2}^\infty)\cap\mathbb{D}_{\rho e^{\varepsilon/2}}\setminus\mathcal{O}$, such that $K_{D_i}(\{a_{iN_i}\}\gg\partial\mathbb{D}_\rho)\leq K_1$. We will show that there exists a Riemann domain $D_i'\subset\mathbb{D}_{\rho e^{\varepsilon/2+\delta}}\setminus\hat{\varphi}(\mathcal{O})\subset\mathbb{C}\setminus\hat{\varphi}(\mathcal{O})$ such that
	$$K_{\hat{\varphi}(D_i')}\left(\hat{\varphi}(a_{iN_i}),\hat{\varphi}_*[L_{ij}']\right)<4K_1.$$
	Consider the isotopy type of $\varphi$ relative to $\mathcal{O}\cap\mathbb{D}_{\rho/2}$ and to those marked points $a_{k(N_k+1)}$ such that $a_{kN_k}\in T$. It contains a map $\xi$ which coincides with $\chi\circ\lambda^{-1}$ on $\mathbb{D}_\rho$ and is equal to identity on $\mathbb{D}_\rho^\infty$ except small disjoint (cylindrical) disks around $a_{k(N_k+1)}$'s (assuming $\delta$ is small). We can assume that the maximal dilatation of $\xi$ on $\mathbb{D}_\rho^\infty$ is smaller than two. Let $\tilde{\xi}$ be the pull-back of $\xi$ normalized as $\tilde{\varphi}$. Then the Riemann domain $D_i':=\tilde{\xi}(D_i)$ is contained in $\mathbb{D}_{\rho e^{\varepsilon/2+\delta}}\setminus\hat{\varphi}(\mathcal{O})$ and, due to Lemma~\ref{lmm:qc_change_of_coordinates},
	$$K_{\hat{\varphi}(D_i')}\left(\hat{\varphi}(a_{iN_i}),\hat{\varphi}_*[L_{ij}']\right)<4K_1.$$
	In other words, $[\tilde{L}_{iN_i}]$ is $\tilde{K}'_{iN_i}$-decomposable, where $\tilde{K}'_{iN_i}<4K_1<\frac{1}{2}(8BK_1K_0^2)^{\beta^0}$.
		
	Now, assume that $j<N_i$ and consider the case when the lift $[\gamma]$ of $[L_{i(j+1)}]$ under $f$ starting at $[\gamma](0)=a_{ij}$ terminates at $[\gamma](1)\in\mathbb{D}_\rho\cap f^{-1}(\partial\mathbb{D}_\rho)$. Then, as in the construction of a standard spider, $[\gamma]$ can be concatenated with some $[\gamma_1]$ such that $[\gamma_1](1)\in\partial\mathbb{D}_\rho$ and the concatenation forms the leg $[\tilde{L}_{ij}]$. Since $[\gamma']:=\hat{\varphi}_*[\gamma]=g^*\left(\varphi_*[L_{i(j+1)}']\right)$ for a holomorphic map $g$, $[\gamma']$ is $K_{ij}'\omega_{ij}$-decomposable by Lemma~\ref{lmm:K_U_for_lifts} and Lemma~\ref{lmm:qc_change_of_coordinates}. Similarly as in last paragraph for $j=N_i$ (except that additionally $\xi$ must fix $f\left([\gamma](1)\right)$), one shows that
	$$K_{\mathbb{C}\setminus\mathcal{O}}\left([\gamma'](0),[\gamma']\right)<4K_1.$$
	Since the concatenation of a decomposable path with any other path remains decomposable, after taking product of the obtained bounds, we get
	$$\tilde{K}_{ij}'<4K_1\omega_{ij}K_{ij}'<4K_1\omega_{ij}(8BK_1K_0^2)^{\beta^{N_i-j-1}}\Omega_{i(j+1)}<$$
	$$<\frac{1}{2}(8BK_1K_0^2)^{\beta^{N_i-j}}\Omega_{ij}.$$
	
	Assume now that $[\gamma](1)\in\mathbb{D}_{\rho}^{\infty}$. In this case $[\tilde{L}_{ij}]$ is formed by concatenating $[\gamma]$, reversed $[\gamma]$ and $[\gamma_\pi]$. After taking $\delta$ small enough, due to Proposition~\ref{prp:K_shifts_imply_distance_bounds} and Remark~\ref{remark:semi-projected} after it, we see that $[\tilde{L}_{ij}]$ is $\tilde{K}_{ij}'$-decomposable where
	$$\tilde{K}_{ij}'<O(\omega_{ij}K_{ij}')^{\beta_1+6},$$
	where $O(.)$ can be made arbitrarily small by making $\delta$ small. Thus
	$$\tilde{K}_{ij}'<\frac{1}{2}(8BK_1K_0^2)^{\beta^{N_i-j}}\Omega_{ij}.$$
	
	Finally, recall that $\hat{\varphi}$ is not identity on $\partial\mathbb{D}_\rho$ but rather $2\delta/3$-distant from it. We might isotope it to make identity on $\partial\mathbb{D}_\rho$ by the price of moving points in $\mathbb{D}_\rho^\infty$ by at most $\delta/3$ and changing the $\tilde{K}_{ij}'$'s by at most a multiplicative factor $2$. This yields the upgraded homeomorphism $\hat{\varphi}$ satisfying properties $(1)-(3)$ and finishes the proof of the theorem.	
\end{proof}

Existence of such $\sigma$-invariant set $\mathcal{I}$ implies in many cases existence of a fixed point of $\sigma$. This is summarized in the following theorem. We stay in the setup of Theorem~\ref{thm:invariant_structure}.

\begin{thm}[Fixed point in $\overline{\mathcal{I}}$]
	\label{thm:fixed_point_existence}
	Let $\mathcal{I}\neq\emptyset$ be the invariant set constructed in Theorem~\ref{thm:invariant_structure}. If, additionally, 
	\begin{enumerate}
		\item $d_{\cyl}(a_{ij},a_{kl})>\varepsilon$ for $a_{ij},a_{kl}\in\mathcal{O}\cap\mathbb{D}_\rho^{\infty}$,
		\item $d_{\cyl}(a_{ij},a_{kl})\to\infty$ for $a_{ij},a_{kl}\in\mathcal{O}\cap\mathbb{D}_r^{\infty}$ as $r\to\infty$,
	\end{enumerate}
	then $\overline{\mathcal{I}}$ contains a fixed point of $\sigma$ and it is a unique fixed point in $\overline{\mathcal{I}}$. 
\end{thm}
\begin{proof}
	Condition $(1)$ implies that every point of $\mathcal{I}$ contains a \qc\ representative and the same is true for the closure $\overline{\mathcal{I}}$ (which is also $\sigma$-invariant).
	
	Condition $(2)$ together with the definition of a separating structure implies that all the marked orbits are either (pre-)periodic or escaping and the orbits of singular values are either strictly pre-periodic or escaping. 
	
	Moreover, from condition $(2)$ follows that every point of $\overline{\mathcal{I}}$ is asymptotically conformal. Therefore, due to Lemma~\ref{lmm:strict_contraction}, there exists an iterate $\sigma^n, n>0$ such that its restriction to $\overline{\mathcal{I}}$ is strictly contracting. Since $\overline{\mathcal{I}}$ is a compact in the \tei\ metric (by sequential compactness argument), $\sigma^n$ must have a fixed point $[\psi]\in\overline{\mathcal{I}}$. From strict contraction of $\sigma^n$ follows that $[\psi]$ is a unique fixed point of $\sigma^n$ in $\overline{\mathcal{I}}$, hence $[\psi]$ is fixed point of $\sigma$ which is unique in $\overline{\mathcal{I}}$.
\end{proof}

Theorem~\ref{thm:fixed_point_existence} has a somewhat restricted usage in the generality it is stated in. The reason is that in practice, in order to obtain the bounds needed for existence of an invariant structure, one might need to assume that $d_{\cyl}(a_{ij},a_{kl})$ tends to $\infty$ with some particular ``speed'' which is rather high.

However, the full generality might be useful if we already have an invariant structure and want to perturb it, e.g., by changing the map $\lambda$ in which we keep bounds for parameters $\rho,q,K_0,K_1,\varepsilon$ and $N(\rho)$ approximately the same. The invariant structure will still exist, but the new map $f$ will have a quite different dynamical behaviour near $\infty$ (e.g., a slower speed of escape).

In some situations it is possible to upgrade Theorem~\ref{thm:fixed_point_existence} by allowing marked points near $\infty$ to come arbitrarily close to each other even in the Euclidean metric. In this setting the nearby points might be a source of non-compactness if they either rotate around each other or nearly collide. This kind of behaviour needs to be controlled separately. See \cite{IDTT3} for the corresponding constructions in the case when $f_0$ is the composition of a polynomial with the exponential map. 

\section{Applications}
\label{sec:applications}
\subsection{Escaping singular values}
\label{subsec:escaping_singular_values}
Instead of proving Theorem~\ref{thm:finite_order} from the Introduction directly, we show that an even stronger statement takes place.

\begin{thm}[Singular values with fast speed of escape]
	\label{thm:esc_singular_orbits}
	Let $f_0$ be a transcendental entire function of finite type having degeneration function $1/\rho^\epsilon$ for some $\epsilon>0$, and satisfying the inequality $\max_{\partial\mathbb{D}_r}\abs{f_0(z)}<\exp^2(\log r)^d$ for some constant $d>1$ and all $r>0$ big enough (it holds, in particular, for functions of finite order).
	
	For a \qc\ map $\lambda:\mathbb{C}\to\mathbb{C}$ equal to identity near $\infty$, consider the quasiregular map $f=\lambda\circ f_0$ with singular values $\{a_{i1}\}_{i=1}^m$ and corresponding escaping singular orbits $\{a_{ij}\}_{j=1}^\infty=\{f^{j-1}(a_{i1})\}_{j=1}^\infty$ such that:
	\begin{enumerate}
		\item for some $\delta>1$ and all $j$ big enough, $\log\log\abs{a_{i(j+1)}}>\delta\log\log\abs{a_{ij}}$,
		\item the set $\{d_{\cyl}(a_{ij},a_{kl}): 0\neq a_{ij}\neq a_{kl}\neq 0\}$, where $d_{\cyl}$ is cylindrical distance, has a positive lower bound.
	\end{enumerate}
	
	Then $f$ is Thurston equivalent to an entire function.
\end{thm}
\begin{proof}[Proof of Theorem~\ref{thm:esc_singular_orbits}]
	Let $D$ be a union of small disks, each around a singular value, so that their closures are pairwise disjoint. Make their radii small enough so that there exists an domain $U\supset\overline{D}$ such that $\mathcal{O}\cap U=\emptyset$. Let $\varepsilon>0$ be the lower bound of the pairwise cylindrical distances between marked point in $\mathbb{D}_1^\infty$. We want to apply Theorem~\ref{thm:invariant_structure} for some big enough $\rho$. Note from the statement of Theorem~\ref{thm:invariant_structure} that $q_0$ depends only on $\varepsilon$. Since for every $i$, $\abs{a_{i(j+1)}}/\abs{a_{ij}}\to\infty$ as $j\to\infty$, for arbitrarily big $\rho$ there is a separating structure $\mathbb{S}[\rho,q,K_0,K_1,\varepsilon]$ where $q<q_0$ and $K_0,K_1$ are some numbers bigger or equal than $1$. In order to turn $\mathbb{S}[\rho,q,K_0,K_1,\varepsilon]$ into an invariant structure from Theorem~\ref{thm:invariant_structure}, it is enough to provide upper bounds on $K_1$ and $N(\rho)$ so that for big enough $\rho$, the product (\ref{eqn:dilatation_per_area_condition}) can be arbitrarily small. Note that in our setting, since we have a fixed $\lambda$, $K_0$ and $\Omega$ are constants and both can be absorbed by $C$ in (\ref{eqn:dilatation_per_area_condition}). 
	
	From Proposition~\ref{prp:log_regularity_finite_order}, we obtain the bound $K_1<C_1(\log\rho)^{2d(m+1)}$ for some constant $C_1>1$ and we can assume that $C_1$ is also absorbed by $C$.
	
	The inequality $\abs{a_{i(j+1)}}>\exp\left(\log\abs{a_{ij}}\right)^\delta$ implies that for big enough $\rho$ and some constant $A>0$, we have a bound $N(\rho)<A\log^3(\rho)$ (note that we use the notation $\log^{n}\rho$ for the $n$-iterate, and $(\log\rho)^d$ for the degree).
	
	Thus, for big enough $\rho$ and a constant $B>0$, 
	$$(CK_1)^{\nu^{N(\rho)}}<\left(C(\log\rho)^{2d(m+1)}\right)^{\nu^{N(\rho)}}<e^{(\log^2\rho)^B}.$$
	Therefore, due to Lemma~\ref{lmm:AAP_for_scaling}, the product
	$$(CK_1)^{\nu^{N(\rho)}}I_q(\rho, D)=O\left(e^{(\log^2\rho)^B}\rho^{-\varepsilon}\right)$$
	tends to $0$ as $\rho$ tends to $\infty$. Condition~\ref{eqn:dilatation_per_area_condition} is satisfied, hence $\mathbb{S}[\rho,q,K_0,K_1,\varepsilon]$ is invariant structure and there exists a $\sigma$-invariant set $\mathcal{I}\subset\hat{\mathcal{T}}_{\mathcal{O}}$ from Theorem~\ref{thm:invariant_structure}.
	
	A simple application of Theorem~\ref{thm:fixed_point_existence} finishes the proof.
\end{proof}

\subsection{Further applications}
\label{subsec:further_applications}
In the last subsection of the article we give a few heuristic explanations about less direct ways to use our constructions.
\subsubsection{Perturbations  of the orbit}
As was mentioned after Theorem~\ref{thm:fixed_point_existence}, in order to construct the invariant set $\mathcal{I}$ of Theorem~\ref{thm:invariant_structure} in practice, we often need $\lambda$ to be fixed in advance so that the singular orbits of $f=\lambda\circ f_0$ behaved well, for instance, to escape ``fast'' and to form a ``sparse'' set.

However, our constructions also allow to construct some functions with a ``slow'' speed of escape of singular values. Let us discuss this on the example of the exponential family.

Let $f_0(z)=e^z$, $\lambda_0=\id$ and $\varepsilon=1$. Of course, the quasiregular function $f=\lambda_0\circ f_0$ is Thurston equivalent (and simply equal) to the entire function $f_0$, but at this point we are interested in perturbing the singular orbit. Denote $a_n:=f^{n-1}(0), n>0$ and consider $\rho=(a_{n+1}+a_n)/2$ for big $n$. From the proof of Theorem~\ref{thm:esc_singular_orbits}, the value of $K_1$ corresponding to the pair $\rho,\varepsilon$ will be expressed as some degree of $\log\rho$ and we have a separating structure $\mathbb{S}[\rho,q,1,K_1,\varepsilon]$ for $f$ with some fixed $q<1/2$ (and with $\lambda=\lambda_0$). Also, it will be invariant in the sense of Theorem~\ref{thm:invariant_structure}. We can perturb this structure by changing $\lambda_0$ to a different $\lambda$, having maximal dilatation $K_0<2$, which is also equal to identity except of a small neighbourhood of the singular value $0$. It is clear that the corresponding new separating structure will still be invariant if $\rho$ is big enough and the new singular orbit is absorbed by $\mathbb{D}_\rho^\infty$ and is $\varepsilon$ sparse on it.

Let us now look more closely at how we can perturb $\lambda_0$. Denote the new singular orbit by $b_n$. Since the derivative of $f_0$ along the orbit $\{a_n\}_{n=1}^\infty$ tends to $\infty$, we might assume that for $N(\rho)<n<k$, $b_n$ is contained in a ``very small'' neighbourhood of $a_n$, for $n=k$, $b_k$ can be any point in a square with horizontal sides, center at $a_k$ and the side length $2\pi$. The image of this square under $f_0$ is the annulus $\mathbb{A}_{e^{a_k-\pi},e^{a_k}+\pi}$. If, for example some ``slow escaping'' orbit of $f_0$ passes through this annulus and is $\varepsilon$-sparse, we can assume that $b_{k+1}$ is a point on it. The obtained separating structure will still be invariant and the existence of the fixed point of the new $\sigma$ will follow by Theorem~\ref{thm:fixed_point_existence}.

In a similar way we can obtain finite singular orbits, though with cycles contained in $\mathbb{D}_\rho^{\infty}$. 

\subsubsection{Captured polynomials}
Even though the techniques developed in this article are suited mainly for dealing with transcendental entire functions, as a byproduct we can also model the polynomials with escaping critical values. Note that the corresponding result is a rather simple subcase of \cite{Cui}, though obtained differently, without considering iterations on infinite-dimensional \tei\ space.

Let $f_0$ be a polynomial of degree $d$ and $\lambda$ be a \qc\ map equal to $\id$ near $\infty$ so that all critical values of $f=\lambda\circ f_0$ escape. Existence of B\"ottcher coordinates for polynomials implies that for every escaping orbit $\{a_n\}_{n=1}^\infty$ of $f_0$, we have asymtotical equality $\abs{a_{n+1}}\sim\abs{a_n}^d$ and $\arg a_{n+1}\sim d\arg a_n$. We can find arbitrarily big $\rho$ so that the first points (on each critical orbit of $f$) which are outside of $\mathbb{D}_\rho$ are $\varepsilon$ distant from each other (in the cylindrical metric). Since $q$ depends only on $\varepsilon$, and $I_q(\rho, D)=0$ whenever we consider $[\varphi]\in\mathcal{T}_{\mathcal{O}}$ with $\varphi|_{\mathbb{D}_{q\rho}^{\infty}}$ being conformal, there is an invariant structure with the given $\lambda$. Note, that the orbits outside of $\mathbb{D}_\rho$ are not necessarily sparse, but we anyway can apply Theorem~\ref{thm:fixed_point_existence} because all $[\varphi]$ as above are asymptotically conformal (because $\varphi$ is conformal near $\infty$). That is, $f$ is Thurston equivalent to a polynomial.

\end{document}